\documentclass{amsart}
\usepackage{color}
\usepackage{amssymb,amsmath,amstext}
\usepackage{graphics}
\usepackage{amsthm}
\usepackage{tkz-graph}
\usepackage{tikz}
\usepackage{url}
\usetikzlibrary{matrix}

\theoremstyle{plain}
\newtheorem{theorem}{Theorem}[section]
\newtheorem{proposition}[theorem]{Proposition}
\newtheorem{corollary}[theorem]{Corollary}
\newtheorem{lemma}[theorem]{Lemma}

\theoremstyle{definition}
\newtheorem{definition}[theorem]{Definition}

\newtheorem{remark}{Remark}

\newtheorem{example}{Example}

\newcommand{\bfx}{\mathbf{x}}

\newcommand{\bbZ}{\mathbb{Z}}
\newcommand{\mcA}{\mathcal{A}}
\newcommand{\mcB}{\mathcal{B}}
\newcommand{\mcC}{\mathcal{C}}
\newcommand{\mcF}{\mathcal{F}}
\newcommand{\mcH}{\mathcal{H}}
\newcommand{\mcK}{\mathcal{K}}
\newcommand{\mcL}{\mathcal{L}}
\newcommand{\mcM}{\mathcal{M}}
\newcommand{\mcP}{\mathcal{P}}
\newcommand{\mfc}{\mathfrak{c}}
\newcommand{\bbK}{\mathbb{K}}
\newcommand{\bbN}{\mathbb{N}}

\DeclareMathOperator{\sign}{sign}
\DeclareMathOperator{\pos}{pos}
\DeclareMathOperator{\Part}{Part}
\newcommand{\nth}{{}^{\textrm{th}}}

\newcommand{\eps}{\varepsilon}

\newcounter{lpnumber}
\newcommand{\outd}{out}
\DeclareMathOperator{\im}{im}
\DeclareMathOperator{\spec}{spec}

\usepackage[pdftex,
            pdfauthor={Madhusudan Manjunath, Frank-Olaf Schreyer, and John Wilmes},
            pdftitle={Minimal Free Resolutions of the G-Parking Function Ideal and the Toppling Ideal},
            pdfsubject={Minimal free resolutions for ideals related to the chip-firing game on graphs},
            pdfkeywords={sandpiles, chip firing, toppling ideal, lattice ideal, minimal free resolution},
            pdfproducer={Latex with hyperref},
            pdfcreator={pdflatex}]{hyperref}

\begin{document}

\title[The $G$-parking Function Ideal and the Toppling Ideal]{Minimal Free Resolutions of the $G$-parking Function Ideal and the Toppling Ideal}
\author{Madhusudan Manjunath}
\address{School of Mathematics, Georgia Institute of Technology,
USA\footnote{Part of the work on this project was done while the first author
was affiliated with the Fachrichtung Mathematik, Universit\"at des Saarlandes,
Germany.}}
\email{mmanjunath3@math.gatech.edu}

\author{Frank-Olaf Schreyer}
\address{Mathematik und Informatik, Universit\"at des Saarlanes, Germany}
\email{schreyer@math.uni-sb.de}

\author{John Wilmes}
\address{Department of Mathematics, University of Chicago, USA}
\email{wilmesj@math.uchicago.edu}

\date{\today}
 
\begin{abstract}
The $G$-parking function ideal $M_G$ of a directed multigraph $G$ is a monomial
ideal which encodes some of the combinatorial information of $G$. It is an
initial ideal of the toppling ideal $I_G$, a lattice ideal intimately related
to the chip-firing game on a graph. Both ideals were first studied by Cori,
Rossin, and Salvy. A minimal free resolution for $M_G$ was given by Postnikov
and Shaprio in the case when $G$ is saturated, i.\,e., whenever there is at
least one edge $(u,v)$ for every ordered pair of distinct vertices $u$ and $v$.
They also raised the problem of an explicit description of the minimal free
resolution in the general case. In this paper, we give a minimal free
resolution of $M_G$ for any undirected multigraph $G$, as well as for a family
of related ideals including the toppling ideal $I_G$.  This settles a
conjecture of Manjunath and Sturmfels, as well as a conjecture of Perkinson and
Wilmes.
\end{abstract}

\maketitle

\section{Introduction}
Let $G$ be a directed multigraph on $n$ vertices with labels in $[n]$. (By
``multigraph'' we mean that every directed edge has a nonnegative integer
weight.)  The adjacency matrix $A_G = (a_{ij})$ of $G$ has rows and columns
indexed by vertices, with $a_{ij}$ the weight of the edge $(i,j)$ if it exists,
and zero otherwise. Let $R =\bbK[x_1,\ldots,x_n]$ be the polynomial ring in
$n$ variables over a field $\bbK$. The $G$-\textbf{parking function ideal} is
\[
M_G = \langle \bfx^{S \to \overline{S}} : S \subset [n-1]\rangle \subset R,
\]
where 
\[
{\bf x}^{S \to \overline{S}}=\prod_{i \in S} x_i^{\sum_{j \notin S} a_{ij}}.
\]
(Note that vertex $n$, which will be the ``sink vertex'' of $G$, never appears
in $S$ but always appears in $\overline{S}$ in the above definition.) The ideal
was first studied by Cori, Rossin, and Salvy \cite{CorRosSal02} in the case of
undirected graphs $G$, and subsequently in the full generality of
directed multigraphs by Postnikov and Shapiro \cite{PosSha}. They gave an
explicit minimal free resolution in the case that $G$ is saturated, i.\,e.,
when the off-diagonal entries of the adjacency matrix are nonzero. In the same
paper, Postnikov and Shapiro asked for an explicit description of the minimal
free resolution in the general case. We resolve the question in this paper for
undirected multigraphs $G$.

We also describe the minimal free resolution of a lattice ideal, called the
toppling ideal $I_G$, associated with the graph $G$. Before defining $I_G$, let
us briefly recall the definition of a lattice ideal.  Let $\mcL$ be a
sublattice of $\bbZ^n$. The lattice ideal $I_{\mcL}$ over $R$ is the ideal
generated by binomials whose exponents differ by a point in $\mcL$. More
precisely, 
\[
I_{\mcL} = \langle \bfx^{\bf u} - \bfx^{\bf v} : {\bf u} -{ \bf v} \in \mcL\rangle,
\]
where $\bfx^{\bf w} = x_1^{w_1}\cdots x_n^{w_n}$ for any ${\bf
w}=(w_1,\dots,w_n) \in \mathbb{N}^{n}$.

Lattice ideals are generalizations of toric ideals and many attempts have been
made to describe their minimal free resolutions.  While explicit descriptions
of nonminimal free resolutions are known, an explicit description of the
minimal free resolution of a lattice ideal is known only in a few cases
\cite[Chapter 9]{MilStu05}. In fact, the well-studied problem of determining
the minimal syzygies of the Veronese embedding of $\mathbb{P}^n$ \cite{Harris}
can be rephrased in terms of minimal free resolutions of lattice ideals.

Toppling ideals, the lattice ideals studied in this paper, are connected to the
Laplacian of a graph. If $G$ is a connected undirected multigraph, the Laplacian matrix of
$G$ is $\Lambda_G = \Delta - A_G$, where $\Delta$ is the diagonal matrix with
entries $\Delta_{ii} = \deg(i) = \sum_j a_{ij}$. The {\bf toppling ideal}
$I_G$ is the lattice ideal defined by the Laplacian lattice $\bbZ^n\cdot
\Lambda_G$. The $G$-parking function ideal and toppling ideal are closely related:
in particular, $M_G$ is an initial ideal of $I_G$ \cite{CorRosSal02}.  An
inhomogeneous version of the toppling ideal was studied by Cori, Rossin, and
Salvy in \cite{CorRosSal02}, and the homogeneous version was subsequently
studied by Perkinson, Perlman, and Wilmes in \cite{Perkinson}, and by Manjunath
and Sturmfels in \cite{ManStu12}. 

The family of toppling ideals appears to be quite rich. Indeed, Perkinson,
Perlman, and Wilmes~\cite{Perkinson} show that any lattice ideal defined by a
full-rank submodule of the root lattice $A_n=\{ \bfx \in \bbZ^{n+1} : \sum_{i =
1}^{n+1} x_i = 0 \}$ is the toppling ideal of some \emph{directed} multigraph.
No effective characterizations of Laplacian lattices of undirected multigraphs
are known, though certainly not every full-rank sublattice of $A_n$ is of this
form \cite{Amini,Perkinson}.  Nevertheless, even in the undirected case
examined in this paper, the ideals $I_G$ and $M_G$ represent broad, natural
categories of ideals.  They are closely tied to the graph chip-firing game, or
abelian sandpile model, first described by Dhar \cite{Dhar90}. The ideal $I_G$
is the lattice ideal corresponding to the lattice of principal divisors of $G$,
and carries information about the Riemann-Roch theory of $G$ (see
\cite{ManStu12,BakNor07}). The study of its minimal free resolution is natural
in this context.

In this paper, we will provide an explicit description of the minimal free
resolution of the ideals $M_G$ and $I_G$ for connected undirected multigraphs
$G$. We verify a conjecture made in \cite[Conjecture 29]{ManStu12} that the
Betti numbers of $I_G$ and $M_G$ coincide (Theorem~\ref{bettinumbers_theo}).
Our construction of the minimal free resolution of $I_G$ also proves
\cite[Conjecture 3.28]{Wil10}, which describes how combinatorial information of
the graph is encoded in the minimal free resolution. A weaker form of the
conjecture, stated in terms of Betti numbers, appeared in \cite[Conjecture
7.9]{Perkinson}, and also follows from Theorem~\ref{bettinumbers_theo}.
Furthermore, as we show in the final section of the paper, the minimal free
resolution of $M_G$ is supported by a CW-complex.  The resolution for $M_G$ is
a Koszul complex when $G$ is a tree, and a Scarf complex when $G$ is saturated
(see \cite{ManStu12}); thus, the $M_G$ form a natural family of ideals whose
minimal free resolutions interpolate between these two extremes.  

We now summarize the results of this paper in terms of the Betti numbers of
$M_G$ and $I_G$.  A connected $k$-partition of $G$ is a partition
$\Pi=\sqcup_{j=1}^k V_j$ of $[n]$ such that the subgraphs induced by $G$ on
each set $V_j$ is connected.  The graph $G_{\Pi}$ associated with the partition
$\Pi$ has vertices the elements of $\Pi$, and the $(V_i,V_j)$-entry of its
adjacency matrix is $\sum_{u \in V_i, v \in V_j} a_{uv}$, where $(a_{uv})$ is
the adjacency matrix for $G$.  Let ${\mathcal P}_k$ be the set of connected
$k$-partitions of $G$ of size $k$. 

\begin{theorem}[Betti Numbers of $M_G$ and $I_G$]\label{bettinumbers_theo}
Let $G$ be an undirected connected multigraph. For a connected $k$-partition
$\Pi$, let $\alpha(\Pi)$ denote the number of acyclic orientations on $G_{\Pi}$
with a unique sink at the set containing vertex $n$.  The Betti numbers of
$M_G$ and $I_G$ are
\[
\beta_k(R/M_G)=\beta_k(R/I_G)=\sum_{\Pi \in \mathcal P_{k+1}} \alpha(\Pi).
\]
\end{theorem}

A few remarks on Theorem \ref{bettinumbers_theo} are in place. The numbers
$\beta_j(R/M_G)$ and $\beta_j(R/I_G)$ do not depend on the choice of the sink
vertex $n$. To obtain a bijection between acyclic orientations of a graph with
unique sink $i$ and acyclic orientations of the same graph with unique sink
$j$, simply reverse the orientations of all edges along paths from $i$ to $j$.
Natural bijections exist between the set of acyclic orientations with unique
sink at a fixed vertex and the set of minimal recurrent configurations of a
graph \cite{BCT}.  Hence, \cite[Conjecture 7.9]{Perkinson} follows from Theorem
\ref{bettinumbers_theo}.

Let us close this section by providing an overview of the proof of Theorem
\ref{bettinumbers_theo}.

\subsubsection*{Overview of proof of Theorem \ref{bettinumbers_theo}.}
We construct complexes $\mcF_1(G)$  and $\mcF_0(G)$ that
are candidates for minimal free resolutions of $I_G$ and $M_G$ respectively
(Sections~\ref{sec:I_G} and~\ref{sec:M_G}). The ranks of the free
modules at each homological degree match the description of the Betti numbers
in Theorem~\ref{bettinumbers_theo}. We are then left with proving the exactness
of $\mcF_0(G)$ and $\mcF_1(G)$. We prove the exactness of $\mcF_0(G)$ in
Section \ref{F0exact_sect}. We exploit the torus action on $\mcF_0(G)$ to
reduce the exactness of $\mcF_0(G)$ to the exactness of certain complexes of
vector spaces (Subsection \ref{subsec:reduction}) and prove the exactness of
these complexes of vector spaces by decomposing them as a direct sum of certain
complexes derived from Koszul complexes (Subsection \ref{subsec:max_star}). The
proof of the exactness of $\mcF_1(G)$ uses the exactness of $\mcF_0(G)$
(Section~\ref{F1exact_sect}). More precisely, we use Gr\"obner degeneration to
derive a family of complexes $\mcF_t(G)$ parametrized by $\spec(\bbK[t])$
such that the fiber at $(t)$ is $\mcF_0(G)$ and the fibers at $(t-t_0)$ for $t_0
\neq 0$ are all isomorphic to $\mcF_1(G)$. The integral weight function realizing
the Gr\"obner degeneration is intimately connected to potential theory on
graphs and one such choice is the function $b_q(D)$ studied in \cite{BakSho11}.
With this information at hand, we use well-known properties of flat
families to deduce that  $\mcF_1(G)$ is exact. \qed

\subsubsection*{Related results.}
Analogous results were obtained simultaneously and independently
by Mohammadi and Shokrieh in~\cite{MohSho12}, using different techniques. They
give the Betti numbers for $M_G$ and $I_G$ and construct minimal free
resolutions by using Schreyer's algorithm to explicitly compute the syzygies.
As well, Horia Mania~\cite{Man12} has given an alternate proof that $\beta_1(R/I_G) =
|\mcP_2|$ by computing the connected components of certain simplicial complexes
associated with $I_G$.

\subsubsection*{Acknowledgements.} We would like to express our gratitude to
Bernd Sturmfels for beginning the collaboration that led to this paper by
kindly hosting the first and third authors at Berkeley. We thank Farbod
Shokrieh for stimulating discussions on potential theory and Gr\"obner bases.
The first author thanks Matt Baker for the support and encouragement in the
course of this work.  The third author was supported in part by NSF Grant
No.~DGE 1144082.

All examples were computed using the computer algebra systems Sage~\cite{sage}
and Macaulay 2~\cite{m2}. Those readers interested in computing their own
examples can find a relevant script on the homepage of the third
author,~\url{http://johnwilmes.name/sand}.

\section{Preliminaries}\label{prelims}

For the rest of this paper, $G$ will denote an undirected connected multigraph with
vertex set $[n]$. In other contexts, it is usual to describe multigraphs as
having ``multiple edges'' joining adjacent vertices, but for our purposes it is
more natural to think of a single edge with nonnegative integer weight. Since
$G$ is undirected, the adjacency matrix $A_G = (a_{ij})$ is symmetric.

We define minimal free resolutions of $M_G$ and $I_G$ in terms of acyclic
partitions.

\begin{definition}{\rm({\bf Acyclic Partition})} 
An {\bf acyclic $k$-partition} $\mcC$ is a pair
$(\Pi,\mathcal{A})$ where $\Pi$ is a connected $k$-partition and
$\mathcal{A}$ is an acyclic orientation on $G_{\Pi}$. We think of $\mcC$ as a
directed graph on $\Pi$. Given a vertex $i$ of $G$, the graph $\mcC$ is an {\bf
$i$-acyclic $k$-partition} if it has a unique sink at the element  of $\Pi$
containing $i$.  If $\mcC$ is an acyclic $k$-partition, we denote by
$\Pi(\mcC)$ the corresponding connected $k$-partition of $G$.
\end{definition}

Acyclic partitions are intimately related to the chip-firing game on a graph.
This game consists of an initial configuration of an integer number $D_j$ of
chips at every vertex $j$ of $G$. Such a configuration is called a
\emph{divisor} (cf.~\cite{BakNor07}), and is viewed as an element $D$ of the free abelian group $\bbZ[V]$,
\[
D = \sum_{j \in V}D_j\cdot j
\]
where $V = [n]$ is the vertex set of $G$. The game is played by \emph{firing} a
vertex $j$, i.\,e., replacing the divisor $D$ with $D - e_j\Lambda_G$, where
$e_j$ is the $j\nth$ standard basis vector.  We say $D$ is \emph{linearly
equivalent} to a divisor $E$ if $E$ can be reached from $D$ by a sequence of
such firings. Thus, viewed as elements of $\bbZ^n$, we have $D$ and $E$
linearly equivalent if and only if they are equivalent modulo the 
Laplacian lattice. For a more thorough introduction to the chip firing
game, the reader is referred to \cite{HLMPPW}.

Let $\mcC$ be an acyclic $k$-partition of $G$, and fix $u \in V(G)$. Let $U \in
\Pi(\mcC)$ be such that $u \in U$. Define $\outd_{\mcC}(u)$ as the number of edges in $G$ between $u$ and
vertices appearing in sets which are out-neighbors of $U$ in $\mcC$, i.\,e.,
\[
\outd_{\mcC}(u) = \sum_{(U,W)\in\mcC}\sum_{w \in W}a_{uw}.
\]
Given an
acyclic $k$-partition $\mcC$ of $G$, we define a divisor $D(\mcC)$ on $G$ by 
\[
D(\mcC) = \sum_{v \in V(G)}\outd_{\mcC}(v)\cdot v.
\]

A $G$-parking function (relative to $n$) is a divisor $D$ on $G$ with $D_n
= -1$ such that if $A \subset V\setminus\{n\}$ and $E$ is the divisor obtained
from $D$ by firing every vertex in $A$, then there is some vertex $i \in A$
such that $E_i < 0$. The divisors $D(\mcC) - 1\cdot n$ for $\mcC$ an
$n$-acyclic $n$-partition are exactly the maximal $G$-parking functions
\cite{BCT}.  If $D$ is a divisor on $G$ with $D_n = -1$, then $D$ is a
$G$-parking function if and only if the monomial $\prod_{i < n}x_i^{D_i}$ is
not in the $G$-parking function ideal $M_G$ \cite{PosSha}.

\begin{definition}{\rm({\bf Chip Firing Equivalence})}
Let $\mcC_1$ and $\mcC_2$ be acyclic $k$-partitions with $\Pi(\mcC_1) =
\Pi(\mcC_2)$, and define the projection $p : \bbZ[V]\to\bbZ[\Pi(\mcC_1)]$ from
divisors of $G$ to divisors of $G_{\Pi(\mcC_1)}$ by $p(D)_U = \sum_{j \in
U}D_j$. Then $\mcC_1$ and $\mcC_2$ are chip firing equivalent, equivalent,
denoted by $\mathcal{C}_1 \sim \mathcal{C}_2$, if the divisors $p(D(\mcC_1))$
and $p(D(\mcC_2))$ on $G_{\Pi(\mcC_1)}$ are linearly equivalent.
\end{definition}

We write $[\mcC]$ for the chip firing equivalence class containing the acyclic
$k$-partition $\mcC$. Each chip firing equivalence of acyclic $n$-partitions
contains a unique $n$-acyclic $k$-partition, as a result of the well-known
equivalence between $G$-parking functions and acyclic orientations
(cf.~\cite{BCT}). The following lemma is an immediate generalization.

\begin{lemma}\label{unique n-acyclic}
Every chip firing equivalence class of acyclic $k$-partitions contains a unique
$n$-acyclic $k$-partition.
\end{lemma}
\begin{proof}
For $\mcC$ an acyclic $k$-partition of $G$, write $\overline{\mcC}$ for the
corresponding acyclic $k$-partition of $G_{\Pi(\mcC)}$. Note that $p(D(\mcC)) =
D(\overline{\mcC})$. Thus, the lemma is immediate from the known result for
acyclic $n$-partitions.
\end{proof}

\begin{remark}\label{rem:edge reversal}
If $\mcC$ is an acyclic $k$-partition with a source at $U \in \Pi(\mcC)$, then
by firing every vertex in $U$ from $D(\mcC)$ we obtain the divisor $D(\mcC')$,
where $\mcC'$ is the acyclic $k$-partition with $\Pi(\mcC') = \Pi(\mcC)$ given
by reversing the orientation of every edge incident on $U$ in $\mcC$, and
preserving all other orientations (c.\,f.~\cite{Goles}). Similarly, we may turn a
sink $U$ into a source by firing all vertices not in $U$. Thus, if $\mcC$ and
$\mcC'$ are acyclic $k$-partitions with $\Pi(\mcC)=\Pi(\mcC')$, then $\mcC \sim
\mcC'$ if $\mcC'$ is obtained from $\mcC$ by iteratively replacing sources with
sinks, or sinks with sources. The converse also holds, by the proof of
Lemma~\ref{lem:projection to partition} below. \qed
\end{remark}

\begin{lemma}\label{lem:projection to partition}
Let $\mcC$ and $\mcC'$ be acyclic $k$-partitions with $\Pi(\mcC) = \Pi(\mcC')$.
Then if $\mcC\sim \mcC'$, the divisors $D(\mcC)$ and $D(\mcC')$ are linearly
equivalent.
\end{lemma}
\begin{proof}
Suppose $p(D(\mcC))$ and $p(D(\mcC'))$ are linearly equivalent. By iteratively
replacing sinks with sources in $\mcC$, as in Remark~\ref{rem:edge reversal},
we obtain an $n$-acyclic $k$-partition, and similarly for $\mcC'$. Thus,
without loss of generality, we may assume $\mcC$ and $\mcC'$ are $n$-acyclic.
But then $\mcC = \mcC'$ by Lemma~\ref{unique n-acyclic}.
\end{proof}

We remark that the converse of Lemma~\ref{lem:projection to partition} also
holds, though we shall not need it.

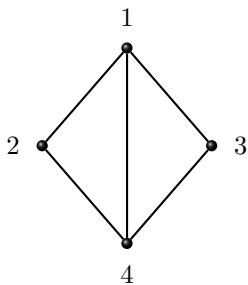
\begin{figure}[ht] 
\begin{center}
\begin{tikzpicture}[scale=1.3]
\SetVertexMath
\GraphInit[vstyle=Art]
\SetUpVertex[MinSize=3pt]
\SetVertexLabel
\tikzset{VertexStyle/.style = {%
shape = circle,
shading = ball,
ball color = black,
inner sep = 1.5pt
}}
\SetUpEdge[color=black]
\Vertex[LabelOut,Lpos=90, Ldist=.1cm,x=4.7,y=1]{1}
\Vertex[LabelOut,Lpos=180, Ldist=.1cm,x=3.833,y=0]{2}
\Vertex[LabelOut,Lpos=0, Ldist=.1cm,x=5.566,y=0]{3}
\Vertex[LabelOut,Lpos=270, Ldist=.1cm,x=4.7,y=-1]{4}
\Edge[](1)(2)
\Edge[](1)(3)
\Edge[](1)(4)
\Edge[](2)(4)
\Edge[](3)(4)
\end{tikzpicture}
\end{center}
\caption{\label{fig:kite graph} The ``kite graph'' on four vertices.}
\end{figure}

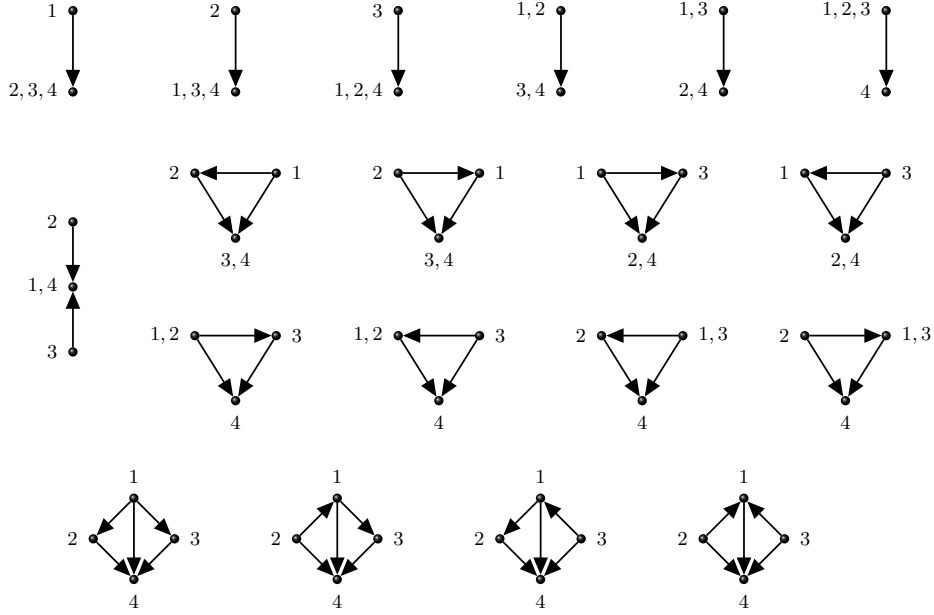
\begin{figure}[ht] 
\begin{center}
\resizebox{\linewidth}{!}{
\begin{tikzpicture}[scale=1.3]
\SetVertexMath
\GraphInit[vstyle=Art]
\SetUpVertex[MinSize=3pt]
\SetVertexLabel
\tikzset{VertexStyle/.style = {%
shape = circle,
shading = ball,
ball color = black,
inner sep = 1.5pt
}}
\SetUpEdge[color=black]
\tikzstyle{every node}=[font=\small]

\Vertex[LabelOut,Lpos=180, Ldist=.05cm,x=0,y=9,L={1}]{a}
\Vertex[LabelOut,Lpos=180, Ldist=.05cm,x=0,y=8,L={2,3,4}]{b}
\Edge[style={-triangle 45}](a)(b)

\Vertex[LabelOut,Lpos=180, Ldist=.05cm,x=2,y=9,L={2}]{a}
\Vertex[LabelOut,Lpos=180, Ldist=.05cm,x=2,y=8,L={1,3,4}]{b}
\Edge[style={-triangle 45}](a)(b)

\Vertex[LabelOut,Lpos=180, Ldist=.05cm,x=4,y=9,L={3}]{a}
\Vertex[LabelOut,Lpos=180, Ldist=.05cm,x=4,y=8,L={1,2,4}]{b}
\Edge[style={-triangle 45}](a)(b)

\Vertex[LabelOut,Lpos=180, Ldist=.05cm,x=6,y=9,L={1,2}]{a}
\Vertex[LabelOut,Lpos=180, Ldist=.05cm,x=6,y=8,L={3,4}]{b}
\Edge[style={-triangle 45}](a)(b)

\Vertex[LabelOut,Lpos=180, Ldist=.05cm,x=8,y=9,L={1,3}]{a}
\Vertex[LabelOut,Lpos=180, Ldist=.05cm,x=8,y=8,L={2,4}]{b}
\Edge[style={-triangle 45}](a)(b)

\Vertex[LabelOut,Lpos=180, Ldist=.05cm,x=10,y=9,L={1,2,3}]{a}
\Vertex[LabelOut,Lpos=180, Ldist=.05cm,x=10,y=8,L={4}]{b}
\Edge[style={-triangle 45}](a)(b)

\Vertex[LabelOut,Lpos=180, Ldist=.05cm,x=0,y=6.4]{2}
\Vertex[LabelOut,Lpos=180, Ldist=.05cm,x=0,y=5.6,L={1,4}]{14}
\Vertex[LabelOut,Lpos=180, Ldist=.05cm,x=0,y=4.8]{3}
\Edge[style={-triangle 45}](2)(14)
\Edge[style={-triangle 45}](3)(14)

\Vertex[LabelOut,Lpos=180, Ldist=.05cm,x=1.5,y=7]{2}
\Vertex[LabelOut,Lpos=0, Ldist=.05cm,x=2.5,y=7]{1}
\Vertex[LabelOut,Lpos=270, Ldist=.05cm,x=2,y=6.2,L={3,4}]{34}
\Edge[style={-triangle 45}](1)(2)
\Edge[style={-triangle 45}](1)(34)
\Edge[style={-triangle 45}](2)(34)

\Vertex[LabelOut,Lpos=180, Ldist=.05cm,x=4,y=7]{2}
\Vertex[LabelOut,Lpos=0, Ldist=.05cm,x=5,y=7]{1}
\Vertex[LabelOut,Lpos=270, Ldist=.05cm,x=4.5,y=6.2,L={3,4}]{34}
\Edge[style={-triangle 45}](2)(1)
\Edge[style={-triangle 45}](2)(34)
\Edge[style={-triangle 45}](1)(34)

\Vertex[LabelOut,Lpos=180, Ldist=.05cm,x=6.5,y=7]{1}
\Vertex[LabelOut,Lpos=0, Ldist=.05cm,x=7.5,y=7,L={3}]{3}
\Vertex[LabelOut,Lpos=270, Ldist=.05cm,x=7,y=6.2,L={2,4}]{24}
\Edge[style={-triangle 45}](1)(3)
\Edge[style={-triangle 45}](1)(24)
\Edge[style={-triangle 45}](3)(24)

\Vertex[LabelOut,Lpos=180, Ldist=.05cm,x=9,y=7]{1}
\Vertex[LabelOut,Lpos=0, Ldist=.05cm,x=10,y=7,L={3}]{3}
\Vertex[LabelOut,Lpos=270, Ldist=.05cm,x=9.5,y=6.2,L={2,4}]{24}
\Edge[style={-triangle 45}](3)(1)
\Edge[style={-triangle 45}](3)(24)
\Edge[style={-triangle 45}](1)(24)


\Vertex[LabelOut,Lpos=180, Ldist=.05cm,x=1.5,y=5,L={1,2}]{12}
\Vertex[LabelOut,Lpos=0, Ldist=.05cm,x=2.5,y=5,L={3}]{3}
\Vertex[LabelOut,Lpos=270, Ldist=.05cm,x=2,y=4.2]{4}
\Edge[style={-triangle 45}](12)(3)
\Edge[style={-triangle 45}](12)(4)
\Edge[style={-triangle 45}](3)(4)

\Vertex[LabelOut,Lpos=180, Ldist=.05cm,x=4,y=5,L={1,2}]{12}
\Vertex[LabelOut,Lpos=0, Ldist=.05cm,x=5,y=5]{3}
\Vertex[LabelOut,Lpos=270, Ldist=.05cm,x=4.5,y=4.2,L={4}]{4}
\Edge[style={-triangle 45}](3)(12)
\Edge[style={-triangle 45}](3)(4)
\Edge[style={-triangle 45}](12)(4)

\Vertex[LabelOut,Lpos=180, Ldist=.05cm,x=6.5,y=5]{2}
\Vertex[LabelOut,Lpos=0, Ldist=.05cm,x=7.5,y=5,L={1,3}]{13}
\Vertex[LabelOut,Lpos=270, Ldist=.05cm,x=7,y=4.2]{4}
\Edge[style={-triangle 45}](13)(2)
\Edge[style={-triangle 45}](2)(4)
\Edge[style={-triangle 45}](13)(4)

\Vertex[LabelOut,Lpos=180, Ldist=.05cm,x=9,y=5]{2}
\Vertex[LabelOut,Lpos=0, Ldist=.05cm,x=10,y=5,L={1,3}]{13}
\Vertex[LabelOut,Lpos=270, Ldist=.05cm,x=9.5,y=4.2]{4}
\Edge[style={-triangle 45}](2)(13)
\Edge[style={-triangle 45}](2)(4)
\Edge[style={-triangle 45}](13)(4)

\Vertex[LabelOut,Lpos=90, Ldist=.05cm,x=0.75,y=3]{1}
\Vertex[LabelOut,Lpos=180, Ldist=.05cm,x=0.25,y=2.5]{2}
\Vertex[LabelOut,Lpos=0, Ldist=.05cm,x=1.25,y=2.5]{3}
\Vertex[LabelOut,Lpos=270, Ldist=.05cm,x=0.75,y=2]{4}
\Edge[style={-triangle 45}](1)(2)
\Edge[style={-triangle 45}](1)(3)
\Edge[style={-triangle 45}](1)(4)
\Edge[style={-triangle 45}](2)(4)
\Edge[style={-triangle 45}](3)(4)

\Vertex[LabelOut,Lpos=90, Ldist=.05cm,x=3.25,y=3]{1}
\Vertex[LabelOut,Lpos=180, Ldist=.05cm,x=2.75,y=2.5]{2}
\Vertex[LabelOut,Lpos=0, Ldist=.05cm,x=3.75,y=2.5]{3}
\Vertex[LabelOut,Lpos=270, Ldist=.05cm,x=3.25,y=2]{4}
\Edge[style={-triangle 45}](2)(1)
\Edge[style={-triangle 45}](1)(3)
\Edge[style={-triangle 45}](1)(4)
\Edge[style={-triangle 45}](2)(4)
\Edge[style={-triangle 45}](3)(4)

\Vertex[LabelOut,Lpos=90, Ldist=.05cm,x=5.75,y=3]{1}
\Vertex[LabelOut,Lpos=180, Ldist=.05cm,x=5.25,y=2.5]{2}
\Vertex[LabelOut,Lpos=0, Ldist=.05cm,x=6.25,y=2.5]{3}
\Vertex[LabelOut,Lpos=270, Ldist=.05cm,x=5.75,y=2]{4}
\Edge[style={-triangle 45}](1)(2)
\Edge[style={-triangle 45}](3)(1)
\Edge[style={-triangle 45}](1)(4)
\Edge[style={-triangle 45}](2)(4)
\Edge[style={-triangle 45}](3)(4)

\Vertex[LabelOut,Lpos=90, Ldist=.05cm,x=8.25,y=3]{1}
\Vertex[LabelOut,Lpos=180, Ldist=.05cm,x=7.75,y=2.5]{2}
\Vertex[LabelOut,Lpos=0, Ldist=.05cm,x=8.75,y=2.5]{3}
\Vertex[LabelOut,Lpos=270, Ldist=.05cm,x=8.25,y=2]{4}
\Edge[style={-triangle 45}](2)(1)
\Edge[style={-triangle 45}](3)(1)
\Edge[style={-triangle 45}](1)(4)
\Edge[style={-triangle 45}](2)(4)
\Edge[style={-triangle 45}](3)(4)

\end{tikzpicture}
}
\end{center}
\caption{\label{fig:kite partitions} The $4$-acyclic $k$-partitions of the
``kite graph'' for $k=2,3,4$.}
\end{figure}

\begin{example}\label{ex:kite graph}
Let $G$ be the ``kite graph'' on four vertices depicted in Figure~\ref{fig:kite
graph}. Then $G$ has a unique acyclic 1-partition, six chip-firing equivalence
classes of acyclic 2-partitions, nine classes of acyclic 3-partitions, and four
classes of acyclic 4-partitions. The 4-acyclic representatives of each of these
is depicted in Figure~\ref{fig:kite partitions}.
\end{example}

\subsection{Edge Contraction}

In order to define the differentials of our free resolutions of $M_G$ and
$I_G$, we will use the operation of edge contraction in acyclic partitions.
For a directed edge $e = (A,B)$ of $\mcC$, we will denote by $e^- = A$ the tail
of $e$, and by $e^+ = B$ the head of $e$. 

\begin{definition}{\rm ({\bf Contractible Edge})} \label{def:contractible}
A (directed) edge $e$ of an acyclic $k$-partition $\mcC$ is {\bf contractible}
if the directed graph $\mcC/e$ given by contracting $e$ (i.\,e., merging the
vertices $e^+$ and $e^-$) is acyclic. The edge $e$ is a contractible edge of
the chip firing equivalence class $\mfc$ if there exists an acyclic
$k$-partition $\mcC \in \mfc$ such that $e$ appears in $\mcC$ and furthermore
$e$ is contractible in $\mcC$.
\end{definition}

Note that by the characterization of chip firing equivalence given in
Remark~\ref{rem:edge reversal}, if $e$ is a contractible edge of a chip firing
equivalence class $\mfc$ and $e$ appears in $\mcC \in \mfc$, then $e$ is
contractible in $\mcC$.

If $\mcC$ is an acyclic $k$-partition and $e$ is a contractible edge of $\mcC$,
then $\mcC/e$ is an acyclic $(k-1)$-partition---the sets $e^-$ and $e^+$ in
$\Pi(\mcC)$ are replaced with the set $e^- \cup e^+$ in $\Pi(\mcC/e)$. If $e$
and $f$ are distinct edges of $\mcC$, and $e$ is contractible, we write $f/e$
for the edge corresponding to $f$ in $\mcC/e$. When the graph we are referring
to is clear from the context, we will sometimes abuse notation and write
$f$ instead of $f/e$. We define $[\mcC]/e$ to be $ [\mcC/e]$. The class $[\mcC]/e$
is well-defined: suppose $\mcC'$ is an acyclic $k$-partition with
$\Pi(\mcC') = \Pi(\mcC)$ such that $e$ is also contractible in $\mcC'$. Let
\[
E = \sum_{u \in e^-}\sum_{v \in e^+} a_{uv}u.
\]
By definition, $D(\mcC/e) = D(\mcC) - E$, and $D(\mcC'/e) = D(\mcC') - E$. Then
$D(\mcC')$ is linearly equivalent to $D(\mcC)$ if and only if $D(\mcC'/e)$ is
linearly equivalent to $D(\mcC/e)$.

\begin{definition}{\rm ({\bf Monomial Associated with Edge Contraction})}
Let $\mfc$ be an equivalence class of acyclic $k$-partitions with contractible
edge $e$. If $e$ appears in $\mcC \in \mfc$, we define the monomial
$m_{\mfc}(e) = {\bf x}^{D(\mcC) - D(\mcC/e)}$.
\end{definition}

The monomial $m_{\mfc}(e)$ is well-defined, since if $e$ also appears in $\mcC'
\in \mfc$, then $D(\mcC) - D(\mcC/e) = D(\mcC') - D(\mcC'/e)$.

\begin{definition} {\rm({\bf Refinements of Acyclic Orientations})}
Let $\ell$, $k \in \mathbb{N}$. An equivalence class $\mfc_1$ of acyclic $k$-partitions is called a {\bf
refinement} of an equivalence class $\mfc_2$ of acyclic $\ell$-partitions if
$\mfc_2$ is obtained from $\mfc_1$ by some sequence of edge contractions.
\end{definition}

Let $\mfc_1$ be an equivalence class of acyclic $k$-partitions and $\mfc_2$
an equivalence class of acyclic $(k-2)$-partitions. Suppose that there exists a
contractible edge $e$ of $\mfc_1$ and a contractible edge $f$ of $\mfc_1/e$
such that $\mfc_2 = (\mfc_1/e)/f$. The following lemma states that we can lift
$f$ to a unique contractible edge of $\mfc_1$, and then $f$ and $e$ are
contractible in either order.

\begin{lemma}\label{mutually contractible}
Let $\mfc_1$ and $\mfc_2$ be equivalence classes of acyclic partitions such
that $\mfc_2 = (\mfc_1/e)/f$ for some edges $e$ and $f$. Then there is a unique
edge $g$ which is contractible relative to $\mfc_1$ such that $g/e = f$.
Furthermore, we have the following:
\begin{enumerate}
    \item $e/g$ is contractible relative to $\mfc_1/g$.
    \item $[(\mfc_1/g)/e] = [(\mfc_1/e)/g] = \mfc_2$.
    \item There exists some $\hat{\mcC} \in \mfc_1$ in which both $e$ and $g$ appear
    (and hence are contractible).
\end{enumerate}
\end{lemma}
\begin{proof}
Let $\mcC \in \mfc_1/e$ be such that $f$ is contractible in $\mcC$, and lift
$\mcC$ to an acyclic partition $\hat{\mcC} \in \mfc_1$ by preserving the orientation
of $e$, and orienting every other edge of $G_{\Pi(\mfc_1)}$ as in $\mcC$. If
there is only one edge $g \in \hat{\mcC}$ such that $g/e = f$, we are done. If
there are two such edges (i.\,e., when either $f^-$ or $f^+$ is equal to $e^+\cup
e^-$), then only one of them is contractible: if $f^- = e^+\cup e^-$ then
$(e^+,f^+)$ is not contractible in $\hat{\mcC}$ but $(e^-,f^+)$ is, and
similarly if $f^+ = e^+\cup e^-$. In all cases, there is exactly one
contractible edge $g \in \hat{\mcC}$ such that $g/e = f$, and the result
follows.
\end{proof}

We now define functions $\sign_\mfc$ for every equivalence class $\mfc$ of acyclic
$k$-partitions, taking the contractible edges of $\mfc$ to $\{\pm 1\}$. We
choose these maps so that if $e$ and $f$ are distinct contractible edges of
$\mfc$, then 
\begin{equation}\label{eq:sign prop1}
\sign_\mfc(e)\sign_{\mfc/e}(f) = -\sign_\mfc(f)\sign_{\mfc/f}(e).
\end{equation}
Furthermore, we insist that if both $e = (A,B)$ and $\hat{e} = (B,A)$ are
contractible edges of $\mfc$ for some sets $A,B\in\Pi(\mfc)$, then
\begin{equation}\label{eq:sign prop2}
\sign_{\mfc}(e) = -\sign_{\mfc}(\hat{e}).
\end{equation}

\begin{proposition}\label{prop:sign exist}
A function $\sign_\mfc$ satisfying~\eqref{eq:sign prop1} and~\eqref{eq:sign
prop2} exists.
\end{proposition}
\begin{proof}
For every equivalence class $\mfc$ of acyclic $k$-partitions, fix a total
ordering $\tau_{\mfc}$ of $\Pi(\mfc)$.  If $\tau$ is a total ordering of a
set $C$, denote by $\pos_{\tau}(c)$ the position of $c$ in this total
ordering for any $c \in C$.

Let $\mfc$ be an equivalence class of acyclic $k$-partitions, and let $e$ be a
contractible edge of $\mfc$. Let $\rho$ be a total ordering of $\Pi(\mfc)$
such that $\pos_{\rho}(e^-) = 0$ and $\pos_{\rho}(e^+) = 1$. Let $\rho/e$
be the total ordering of $\Pi(\mcC/e)$ with $\pos_{\rho/e}(e^- \cup e^+) =
0$, and $A <_{\rho/e} B$ if $A <_{\rho} B$ for all other sets.  Let
$\sign(\rho)$ and $\sign(\rho/e)$ denote the signs of the permutation taking
$\rho$ to $\tau_{\mfc}$, and $\rho/e$ to $\tau_{\mfc/e}$, respectively. We
define $\sign_{\mfc}(e) = \sign(\sigma)\sign(\sigma/e)$.

The function $\sign_{\mfc}$ does not depend on the choice of $\rho$;
indeed, if we take another total ordering $\rho'$ of $\Pi(\mcC)$ for which
$\pos_{\rho'}(e^-) = 0$ and $\pos_{\rho'}(e^+) = 0$, then the sign of the
permutation taking $\rho'$ to $\rho$ is the same as the sign of the
permutation taking $\rho'/e$ to $\rho/e$.

Clearly $\sign_{\mfc}$ satisfies~\eqref{eq:sign prop2}. To verify~\eqref{eq:sign
prop1} for contractible edges $e,f$ of $\mfc$, there are four cases to
consider: (i) $e^-\cup e^+$ and $f^-\cup f^+$ are disjoint; (ii) $e^- = f^-$;
(iii) $e^+ = f^+$; and (iv) $e^- = f^+$. The argument for all four cases is
similar. For example, in case (i) we consider a total ordering $\rho$
of $\Pi(\mfc)$ for which the first four elements are $e^-, e^+, f^-, f^+$, in
that order, and compute signs by contracting $e$ and $f$ in either order.
\end{proof}

\section{Minimal free resolution of $I_G$}\label{sec:I_G}

We now define the complex $\mcF_1(G)$ that, as we will show, is a minimal
free resolution for $I_G$. For an equivalence class of acyclic
$(k+1)$-partitions $\mfc$, let $D(\mfc) \in \bbZ^n/\Lambda_G$ denote the linear
equivalence class of divisors corresponding to the elements of $\mfc$. For $k$ from
$0$ to $n-1$, define the $k\nth$ homological degree of $\mcF_1(G)$ to be the
free module
\[
F_{1,k} = \bigoplus_{\mfc}R(-D(\mfc)),
\]
where the direct sum is taken over all chip-firing equivalence classes of
acyclic $(k+1)$-partitions $\mfc$, which are identified with the standard basis
elements of $F_{1,k}$, and $R(-D(\mfc))$ denotes the twist of $R$
by $-D(\mfc)$. Now we define differentials $\delta_{1,k}: F_{1,k+1} \rightarrow
F_{1,k}$ of $\mcF_1(G)$ by the equations
\begin{equation}\label{diff_form}
\delta_{1,k}(\mfc)=\sum_{e} \sign_{\mfc}(e)m_{\mfc}(e)\cdot (\mfc/e)
\end{equation}
where the sum is taken over contractible edges of $\mfc$. Then $\mcF_1(G)$ is the sequence 
\[
\mcF_1(G): F_{1,n-1} \xrightarrow{\delta_{1,n-1}}\cdots 
\xrightarrow{\delta_{1,2}}F_{1,1} \xrightarrow{\delta_{1,1}} F_{1,0}.
\]

\begin{example}\label{IG_ex}
For the ``kite graph'' $G$ depicted in Figure~\ref{fig:kite graph}, the
complex $\mcF_1(G)$ reads as follows:
\[
\mcF_1(G): R^4  \xrightarrow{\delta_{1,3}} R^9
\xrightarrow{\delta_{1,2}}  {R}^6    \xrightarrow{\delta_{1,1}}   {R}^1.
\]
The matrices of differentials are
\begin{align*}
\delta_{1,1} & = 
\left(\begin{array}{rrrrrr}
x_{1}^{3} -  x_{2} x_{3} x_{4} & x_{2}^{2} -  x_{1} x_{4} & x_{3}^{2} -  x_{1}
x_{4} & x_{1}^{2} x_{2} -  x_{3} x_{4}^{2} & x_{1}^{2} x_{3} -  x_{2} x_{4}^{2}
& x_{1} x_{2} x_{3} -  x_{4}^{3}
\end{array}\right)\\
\delta_{1,2} &= \left(\begin{array}{rrrrrrrrr}
0 & - x_{2} & - x_{4} & - x_{3} & - x_{4} & 0 & 0 & 0 & 0 \\
- x_{3}^{2} + x_{1} x_{4} & - x_{3} x_{4} & - x_{1}^{2} & 0 & 0 & 0 & 0 & -
x_{4}^{2} & - x_{1} x_{3} \\
x_{2}^{2} -  x_{1} x_{4} & 0 & 0 & - x_{2} x_{4} & - x_{1}^{2} & - x_{4}^{2}
& - x_{1} x_{2} & 0 & 0 \\
0 & x_{1} & x_{2} & 0 & 0 & - x_{3} & - x_{4} & 0 & 0 \\
0 & 0 & 0 & x_{1} & x_{3} & 0 & 0 & - x_{2} & - x_{4} \\
0 & 0 & 0 & 0 & 0 & x_{1} & x_{3} & x_{1} & x_{2}
\end{array}\right)\\
\delta_{1,3} &=\left(\begin{array}{rrrr}
- x_{4} & 0 & 0 & x_{1} \\
x_{3} & 0 & - x_{4} & 0 \\
0 & x_{3} & 0 & x_{4} \\
- x_{2} & - x_{4} & 0 & 0 \\
0 & 0 & x_{2} & - x_{4} \\
x_{1} & x_{2} & 0 & 0 \\
0 & 0 & - x_{1} & x_{2} \\
- x_{1} & 0 & x_{3} & 0 \\
0 & - x_{1} & 0 & - x_{3}
\end{array}\right).
\end{align*}
   
The basis elements of the free modules in $\mcF_1(G)$ correspond to
the six chip firing equivalence classes of acyclic $2$-partitions, nine chip
firing equivalence classes of acyclic $3$-partitions and four chip firing
equivalence classes of acyclic $4$-partitions, in the order (from left to
right) depicted in Figure~\ref{fig:kite partitions}.  \qed
\end{example}

The first main result of this paper is the following:
\begin{theorem}\label{thm:I_G}
$\mcF_1(G)$ is a minimal free resolution of $I_G$.
\end{theorem}

We remark that $\mcF_1(G)$ is naturally graded by $\bbZ^n/\Lambda_G$. In Lemma
\ref{lem:I_G complex} below, we will show that $\mcF_1(G)$ is complex and that
the cokernel of $\delta_{1,1}$ is equal to $R/I_G$. We will complete the proof
of Theorem \ref{thm:I_G} in Section \ref{F1exact_sect} where we establish the
exactness of $\mcF_1(G)$.

\begin{lemma}\label{lem:I_G complex}
$\mcF_1(G)$ is a complex of free $R$-modules, and the cokernel of
$\delta_{1,1}$ is equal to $R/I_G$.
\end{lemma}
\begin{proof}
First, we show that the cokernel of $\delta_{1,1}$ is equal to $R/I_G$. If $\mfc$
is an equivalence class of acyclic 2-partitions, then the two elements $\mcC_1,
\mcC_2$ of $\mfc$ are the two possible orientations of an edge
$\{A,\overline{A}\}$, where both $A$ and $\overline{A}$ induce connected
subgraphs of $G$.  Then 
\[
\delta_{1,1}(\mfc) = \pm(\bfx^{D(\mcC_1)} - \bfx^{D(\mcC_2)}) =
\pm(\bfx^{S \to \overline{S}} - \bfx^{\overline{S}\to S})
\]
which lies in $I_G$ since $D(\mcC_1)$ and $D(\mcC_2)$ are linearly equivalent.
Furthermore, by \cite[Theorem 25]{ManStu12}, (following \cite[Theorem
14]{CorRosSal02}), the binomials 
\[
{\bf x}^{S \to \overline{S}} - {\bf x}^{\overline{S} \to S},
\]
where both $S$ and $\overline{S}$ are connected, form a Gr\"obner basis for
$I_G$, and in particular they generate $I_G$.

Now we show that the $\delta_{1,k}$ are differentials.  Fix an equivalence class
$\mfc$ of acyclic $k$-partitions of $G$, with $k \ge 2$. We wish to show that
for any equivalence class $\mfc'$ of acyclic $(k-2)$-partitions of $G$, the
$\mfc'$ component of $\delta_{1,k-1}(\delta_{1,k}(\mfc))$ is $0$. A nonzero term
appearing in the $\mfc'$ component of $\delta_{1,k-1}(\delta_{1,k}(\mfc))$ results
from a sequence of two edge contractions, say a contractible edge $e$ of $\mfc$
and a contractible edge $f$ of $\mfc/e$. By Lemma~\ref{mutually contractible},
there exists a unique edge $g$ of $\mfc$ such that $g$ is contractible and $g/e
= f$. Furthermore, $\mfc' = (\mfc/g)/e$. 
Thus, it suffices to show that
\[
\sign_{\mfc}(e)\sign_{\mfc/e}(g)m_{\mfc}(e)m_{\mfc/e}(g) +
\sign_{\mfc}(g)\sign_{\mfc/g}(e)m_{\mfc}(g)m_{\mfc/g}(e) = 0.
\]
By Property~\eqref{eq:sign prop1} of $\sign_{\mfc}$, it suffices to show that
\[
m_{\mcC}(e)m_{\mcC/e}(g) = m_{\mcC}(g)m_{\mcC/g}(e).
\]

Let $\mcC \in \mfc$ be such that both $e$ and $g$ appear in $\mcC$. 
Note that such an acyclic $k$-partition is guaranteed by Lemma~\ref{mutually contractible}.  We have
\[
m_{\mcC}(e)m_{\mcC/e}(g) = {\bf x}^{D(\mcC) - D(\mcC/e)}{\bf x}^{D(\mcC/e) - D((\mcC/e)/g)}
= {\bf x}^{D(\mcC) - D((\mcC/e)/g)}
\]
and similarly $m_{\mcC}(g)m_{\mcC/g}(e) = {\bf x}^{D(\mcC) - D((\mcC/e)/g)}$.
\end{proof}

\begin{remark}\label{binsyz_rem}
 When viewed as a matrix with entries over the polynomial ring
$R$, the nonzero entries of $\delta_{1,k}$ are either monomials or binomials. Let
$\mfc$ be an equivalence class of acyclic $(k+1)$-partitions, and $\mfc'$ an
equivalence class of acyclic $k$-partitions. Suppose some contractible edge $e$
of $\mfc$ satisfies $\mfc/e = \mfc'$. Then $\Pi(\mfc')$ is obtained from
$\Pi(\mfc)$ by replacing $e^-$ and $e^+$ with $e^-\cup e^+$. Hence, 
there is at most one other directed edge $\hat{e}$ such that $\mfc/\hat{e} =
\mfc'$, namely $\hat{e} = (e^+,e^-)$.

In fact, $\mfc/e = \mfc/\hat{e}$ if and only if the edge $\{e^-,e^+\}$ of
$\mfc$ is a bridge of $G_{\Pi(\mfc)}$. If the edge is a bridge between sets $A$
and $B$, and the edge is oriented from $A$ to $B$ in some acyclic
$(k+1)$-partition $\mcC \in \mfc$, then we can fire $A$ to obtain a chip firing
equivalent acyclic $(k+1)$-partition which differs from $\mcC$ only on the
orientation of this edge. It follows that $\mfc/e = \mfc/\hat{e}$. Similarly,
if $\mfc/e = \mfc/\hat{e} = \mfc'$, then fix $\mcC \in \mfc'$. We can lift
$\mcC$ to $\mfc$ by introducing the edge $\{e^-,e^+\}$ in either
orientation, so the two resulting acyclic $(k+1)$-partitions are chip
firing equivalent. By Remark~\ref{rem:edge reversal}, we may obtain one acyclic
$(k+1)$-partition from the other by iteratively turning sinks into sources by
reversing edges. Since no edge other than $e$ is reversed at the end of this
process, it follows that $e$ does not participate in any cycles. Thus, binomial
entries in $\delta_{1,k}$ correspond to bridges of partition graphs, i.\,e., cuts of $G$.
\end{remark}

\section{Minimal free resolution of $M_G$}\label{sec:M_G}

We will now define the complex $\mcF_0(G)$, the minimal free resolution for
$M_G$. As we will show later in this section, in the case when $G$ is a tree,
$\mcF_0(G)$ is a Koszul complex.

For $k$ from $0$ to $n-1$, define the $k\nth$ homological degree of $\mcF_0(G)$
to be the free module
\[
F_{0,k} = \bigoplus_{\mcC}R(-D(\mcC))
\]
where the direct sum is taken over all $n$-acyclic $(k+1)$-partitions $\mcC$ of
$G$, which are identified with the standard basis elements of $F_{0,k}$. Now we
define differentials $\delta_{0,k}: F_{0,k} \rightarrow F_{0,k-1}$ of
$\mcF_0(G)$ by the equations
\begin{equation}\label{diff_form M_G}
\delta_{0,k}(\mcC)=\sum_{e} \sign_{\mcC}(e)m_{\mcC}(e)\cdot (\mcC/e)
\end{equation}
where the sum is taken over contractible edges of $\mcC$. Then $\mcF_0(G)$ is
the sequence
\[
\mcF_0(G): F_{0,n-1} \xrightarrow{\delta_{0,n-1}}\cdots 
\xrightarrow{\delta_{0,2}}F_{0,1} \xrightarrow{\delta_{0,1}} F_{0,0}.
\]

We emphasize that the essential difference between $\mcF_1(G)$ and $\mcF_0(G)$
is that in the differentials of the latter, the sum is taken only over
contractible edges of an $n$-acyclic $(k+1)$-partition, rather than over all
contractible edges of a chip-firing equivalence class of acyclic partitions.
Since edges only appear in one orientation in the $n$-acyclic partition of
chip-firing equivalence class, the condition in Equation~\eqref{eq:sign prop2}
is no longer relevant. For the maps $\sign_{\mcC}$, we require
only Property~\eqref{eq:sign prop1} to hold.  Finally, we note that whereas $\mcF_1(G)$
was graded by $\bbZ^n/\Lambda_G$, the complex $\mcF_0(G)$ is graded by
$\bbN^{n-1}$.

\begin{example}\label{MG_ex}
For the ``kite graph'' $G$ depicted in Figure~\ref{fig:kite graph}, the complex
$\mcF_0(G)$ reads as follows:
\[
\mcF_0(G): {R}^4  \xrightarrow{\delta_{0,3}} {R}^9
\xrightarrow{\delta_{0,2}}  {R}^6    \xrightarrow{\delta_{0,1}}   {R}^1.
\]
The matrices of differentials are
\begin{align*}
\delta_{0,1} &=\left(\begin{array}{rrrrrr}
x_{1}^{3} & x_{2}^{2} & x_{3}^{2} & x_{1}^{2} x_{2} & x_{1}^{2} x_{3} & x_{1}
x_{2} x_{3}
\end{array}\right)\\
\delta_{0,2} &=\left(\begin{array}{rrrrrrrrr}
0 & - x_{2} & 0 & - x_{3} & 0 & 0 & 0 & 0 & 0 \\
- x_{3}^{2} & 0 & - x_{1}^{2} & 0 & 0 & 0 & 0 & 0 & - x_{1} x_{3} \\
x_{2}^{2} & 0 & 0 & 0 & - x_{1}^{2} & 0 & - x_{1} x_{2} & 0 & 0 \\
0 & x_{1} & x_{2} & 0 & 0 & - x_{3} & 0 & 0 & 0 \\
0 & 0 & 0 & x_{1} & x_{3} & 0 & 0 & - x_{2} & 0 \\
0 & 0 & 0 & 0 & 0 & x_{1} & x_{3} & x_{1} & x_{2}
\end{array}\right)\\
\delta_{0,3} &=\left(\begin{array}{rrrr}
0 & 0 & 0 & x_{1} \\
x_{3} & 0 & 0 & 0 \\
0 & x_{3} & 0 & 0 \\
- x_{2} & 0 & 0 & 0 \\
0 & 0 & x_{2} & 0 \\
x_{1} & x_{2} & 0 & 0 \\
0 & 0 & - x_{1} & x_{2} \\
- x_{1} & 0 & x_{3} & 0 \\
0 & - x_{1} & 0 & - x_{3}
\end{array}\right).
\end{align*}

The basis elements of the free modules in $\mcF_1(G)$ correspond to
the six $4$-acyclic $2$-partitions, nine $4$-acyclic $3$-partitions, and four
$4$-acyclic $4$-partitions, in the order (from left to right) depicted in
Figure~\ref{fig:kite partitions}.  \qed
\end{example}

The second main result of this paper is that $\mcF_0(G)$ is a minimal free
resolution of $M_G$.

\begin{theorem}\label{thm:M_G}
$\mcF_0(G)$ is a minimal free resolution of $M_G$.
\end{theorem}
The proof that  $\mcF_0(G)$ is a complex of free $R$-modules and the
cokernel of $\delta_{0,1}$ is equal to $R/M_G$ proceeds along the lines of  the
analogous statement for  $\mcF_1(G)$.  We will complete the proof of
Theorem \ref{thm:M_G} in Section \ref{F0exact_sect} where we establish the
exactness of  $\mcF_0(G)$. In the rest of this section, we will show that
for any tree $T$, the complex $\mcF_0(T)$ is isomorphic to the Koszul complex.
Since $M_T$ is the irrelevant ideal $\langle x_1,\dots,x_{n-1} \rangle$ it
follows that $\mcF_0(T)$ is a minimal free resolution of $M_T$.

\begin{lemma}\label{lem:koszul sign}
Let $R = \bbK[x_1,\ldots,x_n]$ and let
\[
\mcK: K_n \xrightarrow{\delta_n} \cdots \xrightarrow{\delta_2} K_1
\xrightarrow{\delta_1} K_0
\]
be the Koszul complex for $(x_1,\ldots,x_n)$. Suppose $\delta_i': K_{i+1} \to
K_i$ are differentials which agree with $\delta_i$ as monomial matrices up to
the signs of their entries. Then for some collection $B$ of basis elements of
each of the $K_i$, we can obtain $\delta$ from $\delta'$ by composing with the
map sending each element of $B$ to its negative.  In particular, the complex
given by the $\delta_i'$ is exact.
\end{lemma}
\begin{proof}
By induction on $i$. Suppose $\phi:K_{i-1}\to K_{i-1}$ is given by sending some
collection of basis elements to their negative, and $\delta_{i-1}'\circ\phi =
\delta_{i-1}$. Write $\phi\circ\delta_i'$ and $\delta_i$ as matrices over $R$.
We wish to show that these matrices are equal up to multiplying some collection
of columns by $-1$.  Let $u$ be the column of $\phi\circ\delta_i'$
corresponding to basis element $e$, and $v$ the $e$-column of $\delta_i$. Since
$\delta_{i-1}'\circ\phi u = \delta_{i-1}u = \vec{0}$, and since $\mcK$ is
exact, it follows that $u$ is an $R$-linear combination of the columns of
$\delta_i$. But the nonzero entries of $u$ are monomials of degree one, which
agree with $v$ up to their sign. Furthermore, by the definition of the Koszul
complex, every variable $x_i$ appears at most once in each row of $\delta_i$.
It follows that $u = \pm v$.
\end{proof}

\begin{lemma}\label{thm:tree}
Let $G$ be a tree. Then $M_G = \langle x_1,\ldots,x_{n-1}\rangle$, and
$\mcF_0(G)$ is a minimal free resolution of $R/M_G$. (In fact,
$\mcF_0(G)$ is a Koszul complex.) 
\end{lemma}
\begin{proof}
If $G$ is a tree, and $\mcC$ is an $n$-acyclic $k$-partition, then let
$\Delta_\mcC$ denote the subset of $\{1,\ldots,n-1\}$ given by those vertices
of $G$ with an out-edge appearing in $\mcC$. Then in characteristic two,
$\mcF_0(G)$ is isomorphic to the Koszul complex via the map sending
$e_{\mcC}$ to the basis element of the Koszul complex corresponding to
$\Delta_{\mcC}$. Use Lemma \ref{lem:koszul sign} and the fact that $R/{\langle
x_1,\ldots,x_{n-1}\rangle}$ is minimally resolved by the Koszul complex to
conclude that $\mcF_0(G)$ minimally resolves $R/M_G$.
\end{proof}

\section{Exactness of $\mcF_0(G)$}\label{F0exact_sect}

In this section we will establish the exactness of $\mcF_0(G)$. In
Subsection \ref{subsec:reduction}, we reduce the  exactness of
$\mcF_0$ to the exactness of certain complexes of vector spaces and in
Subsection \ref{subsec:max_star}, we show that these complexes of vector spaces
are exact. 

\subsection{Reduction to a Complex of Vector Spaces}\label{subsec:reduction}
For a prime ideal $P$ of $R$, denote by $\kappa(P)$ the residue field $R_P/P$
at $P$. 

 \begin{lemma}\label{vectred_theo} Let $P_j$ be the ideal $\langle
 x_1,\dots,x_{j-1},x_{j+1},\dots, x_{n-1} \rangle$ of
 $\mathbb{K}[x_1,\dots,x_{n-1}]$. For the complex $\mcF_0$, the following statements
 are equivalent: 
 \begin{enumerate}
 \item $\mcF_0$ is exact. 
\item  The complexes $(\mcF_0)_{P}$  and $(\mcF_0)_{P} \otimes \kappa(P)$ are
split exact for all prime ideals $P$ of $\mathbb{K}[x_1,\dots,x_{n-1}]$ except
the irrelevant ideal $\langle x_1,\dots,x_{n-1}\rangle$.
\item For each $j$ from $1$ to $n-1$, the complex $(\mcF_0)_{P_j}
\otimes \kappa(P_j)$ is split exact  as a complex of vector spaces.
\end{enumerate}
 \end{lemma}
 \begin{proof}
 
($1 \Rightarrow 2$) Since exactness is a local property and $M_G$ is an Artinian
monomial ideal, the complex $(\mcF_0)_{P}$ is,  in fact,  split exact for all prime
ideals $P$.  Hence,  for all $i \ge 1$ the modules ${\rm Tor}^{i}((\mcF_0)_{P},\kappa(P))$
are zero. This shows that  $(\mcF_0)_{P}  \otimes \kappa(P)$ is also split exact. 

\noindent ($2 \Rightarrow 3$)   Note that $P_j$ is a prime ideal for all integers $j$ from $1$ to $n$. 

\noindent ($3 \Rightarrow 1$)
We first show that if $(\mcF_0)_{P_j}  \otimes \kappa(P_j)$ is split
exact  as a complex of vector spaces over the residue field $\kappa(P_j)$ of
the local ring at $P_j$, then $(\mcF_0)_{P_j}$ is split exact by the
following argument.

Take an element  $b$ in $(F_{0,0})_{P_{j}}$ and consider its projection $b_p$
in  $(F_{0,0})_{P_{j}} \otimes \kappa(P_j)$. Since $(\mcF_0)_{P_j}
\otimes \kappa(P_j)$ is split exact, $b_p$ is contained image of the first
differential of $(\mcF_0)_{P_{j}} \otimes \kappa(P_j)$. Using the map between the
complexes $(\mcF_0)_{P_j}$ and $(\mcF_0)_{P_{j}} \otimes \kappa(P_j)$
and the fact that the natural projection from $(F_{0,0})_{P_{j}}$ to
$(F_{0,0})_{P_{j}} \otimes \kappa(P_j)$ is surjective, we know that there is an
element $c$ in $F_{0,1}$ such that $\delta_{0,1}(c)$ is equal to $b$ modulo
$\mathfrak{m}\cdot (F_{0,0})_{P_{j}}$ where $\mathfrak{m}$ is the unique
maximal ideal of the local ring $R_{P_j}$.  Hence, $
\im\delta_{0,1}+\mathfrak{m}\cdot (F_{0,0})_{P_{j}}=(F_{0,0})_{P_{j}}$.  By
Nakayama's lemma, $\im\delta_{0,0}=(F_{0,0})_{P_{j}}$ and hence,
$\delta_{0,0}$ is surjective with $\ker\delta_{0,1} \oplus
(F_{0,0})_{P_{j}}=(F_{0,1})_{P_j}$. This shows that $\ker\delta_{0,1}$
is a projective module over a local ring and hence $\ker\delta_{0,1}$
is also free.  Thus, we can write $(\mcF_0)_{P_j}$  as a direct sum of
a trivial complex $(F_{0,0})_{P_j} \xrightarrow{id} (F_{0,0})_{P_j}$ and
another complex $\mathcal{F'}$.  We iterate the argument on $\mathcal{F'}$ to
deduce that $(\mcF_0)_{P_j}$ is a trivial complex and is therefore split exact.

We now suppose that $(\mcF_0)_{P_j}$ is split exact and show that the
sheafified complex $\tilde{\mcF_0}$ (over  Proj$R$) is exact. Using
the open property of exactness, we deduce that the set $L$ of points
$\mathbb{P}^{n-2}$ whose stalks are not exact is a Zariski closed set.
Since $L$ is Zariski closed and since the module $M_G$ is invariant under
the torus action, $L$ is also invariant under the action of the algebraic torus 
$(\overline{\mathbb{K}}^{*})^{n-1}$.  Every closed set invariant under the torus
action is the zero set of a monomial ideal and hence, if nonempty, it contains
one of the coordinate points $e_j$, and indeed the vanishing ideal of $e_j$ is
$P_j$.  Hence, we deduce that the stalk of $\tilde{\mcF_0}$ at every closed
point of $\mathbb{P}^{n-2}$ is exact and hence $\tilde{\mcF_0}$ is exact. Thus,
the support of the homology modules of $\mcF_0$ is either empty or the
irrelevant maximal ideal $\langle x_1,\dots,x_n \rangle$. Hence, the Krull
dimension of the nonzero homology module of $\mcF_0$ is zero (since the Krull
dimension of a module is by definition the dimension of its support).  Since
the depth is at most the Krull dimension, the depth of the homology modules of
$\mcF_0$ are all zero.  We now note the hypothesis for the acyclicity lemma of
Peskin and Szpiro \cite[Lemma 20.11]{Eisen95} are all satisfied:  we have a
graded ring $R$,  the complex $\mcF_0$ has length at most the depth of $R$ and
the depth of $F_{0,k}$ is at least $k$ (actually in this case $F_{0,k}$ has
depth $n$). Hence, we deduce that the homology modules of $\mcF_0$ are all zero
and $\mcF_0$ is exact.
\end{proof} 

Note that the matrix of the differentials of $(\mcF_0)_{P_j}  \otimes
\kappa(P_j)$ can be obtained by substituting one for $x_j$ and zero for all
other indeterminates in the corresponding matrix of differentials of $\mcF_0$. 

\subsection{Exactness of the Complex of Vector Spaces}\label{subsec:max_star}

In this section, we show that $(\mcF_0)_{P_j} \otimes \kappa(P_j)$ is exact
for any $j$ from $1$ to $n-1$. To that end, we will decompose
$(\mcF_0)_{P_j} \otimes \kappa(P_j)$ into a direct sum of complexes of
vector spaces arising from localized Koszul complexes. In fact, the Koszul complexes
will be the complexes $\mcF_0(H)$ for certain star graphs $H$ which
we will now define.

Consider the differential $\delta_{k,j}^* = (\delta_{0,k})_{P_j} \otimes \kappa(P_j)$
induced from the differential $\delta_{0,k}$ of $\mcF_0$.  We represent
$\delta_{k,j}^*$ as a zero-one matrix with respect to the
standard basis elements with every one in the matrix corresponding to merging
vertices $V_r$ and $V_q$ of an $n$-acyclic $(k+1)$-partition $\mcC$ such that: (i)
$j \in V_r$, and (ii) any edge of $G$ with one vertex in $V_r$ and the other
vertex in $V_q$ is incident on $j$. Say an edge $(V_r,V_q)$ of $\mcC$ is a
\emph{$j$-edge} if it is contractible, and satisfies (i) and (ii), and consider
the subgraph $\mathcal{S}$ of $\mcC$ composed of the $j$-edges, and all vertices incident
with these edges.  Thus, $\mathcal{S}$ is a star, and every nonzero entry in the
$e_\mcC$-column of $\delta_{k,j}^*$ corresponds to an edge of
$\mathcal{S}$. We say $\mathcal{S}$ is the \emph{$j$-star} associated with $\mcC$, and denote it by
$\mathcal{S}(\mcC)$. It is important to clarify that $\mcC$ is part of the data of
$\mathcal{S}(\mcC)$; even if the $j$-edges of two distinct $n$-acyclic
$(k+1)$-partitions $\mcC$ and $\mcC'$ are identical, we do not identify
$\mathcal{S}(\mcC)$ with $\mathcal{S}(\mcC')$.

If $e_\mcC$ appears in a term in the image of $\delta_{k+1,j}^*$, it follows
that $\mcC$ is obtained by contracting a $j$-edge of some $n$-acyclic
$(k+1)$-partition $\mcC'$. In fact, there is a unique maximal such refinement
of $\mcC$. 

\begin{proposition}\label{prop:max j refinement}
Let $\mcC$ be an $n$-acyclic $k$-partition, and let $\ell$ be maximal such that
there exists an $n$-acyclic $\ell$-partition $\mcC'$ with $\mcC$ obtained from
$\mcC'$ be a sequence of contractions of $j$-edges. Then $\mcC'$ is the unique
such $n$-acyclic $\ell$-partition.
\end{proposition}

Now suppose $\mcC'$ is the unique maximal refinement of $\mcC$ by $j$-edges, as
in Proposition~\ref{prop:max j refinement}. If $\mcC' = \mcC$, we say the star
$S(\mcC)$ is \emph{maximal}. In any case, if $\mcC$ is obtained from $\mcC'$ by
contracting a collection $E$ of edges, then $\mathcal{S}(\mcC)$ is obtained from
$\mathcal{S}(\mcC')$ by contracting the same collection of edges.

\begin{figure}[ht] 
\begin{center}
\resizebox{\linewidth}{!}{
\begin{tikzpicture}[scale=1.3]
\definecolor{lgray}{rgb}{0.7,0.7,0.7}
\SetVertexMath
\GraphInit[vstyle=Art]
\SetUpVertex[MinSize=3pt]
\SetVertexLabel
\SetUpEdge[color=lgray]
\tikzstyle{every node}=[font=\small]

\tikzset{VertexStyle/.style = {%
shape = circle,
shading = ball,
ball color = black,
inner sep = 1.5pt
}}
\Vertex[LabelOut,Lpos=180, Ldist=.05cm,x=0,y=5,L={1}]{a}
\Vertex[LabelOut,Lpos=180, Ldist=.05cm,x=0,y=4.2,L={2,3,4}]{b}
\Edge[color=black,style={-triangle 45}](a)(b)

\Vertex[LabelOut,Lpos=0, Ldist=.05cm,x=2.5,y=5]{1}
\Vertex[LabelOut,Lpos=180, Ldist=.05cm,x=1.5,y=5]{2}
\tikzset{VertexStyle/.style = {%
shape = circle,
shading = ball,
ball color = lgray,
inner sep = 1.5pt
}}
\Vertex[LabelOut,Lpos=270, Ldist=.05cm,x=2,y=4.2,L={3,4}]{34}
\Edge[color=black,style={-triangle 45}](1)(2)
\Edge[style={-triangle 45}](1)(34)
\Edge[style={-triangle 45}](2)(34)

\Vertex[LabelOut,Lpos=180, Ldist=.05cm,x=4,y=5]{2}
\tikzset{VertexStyle/.style = {%
shape = circle,
shading = ball,
ball color = black,
inner sep = 1.5pt
}}
\Vertex[LabelOut,Lpos=0, Ldist=.05cm,x=5,y=5]{1}
\Vertex[LabelOut,Lpos=270, Ldist=.05cm,x=4.5,y=4.2,L={3,4}]{34}
\Edge[style={-triangle 45}](2)(1)
\Edge[style={-triangle 45}](2)(34)
\Edge[color=black,style={-triangle 45}](1)(34)

\Vertex[LabelOut,Lpos=180, Ldist=.05cm,x=6.5,y=5]{1}
\Vertex[LabelOut,Lpos=0, Ldist=.05cm,x=7.5,y=5,L={3}]{3}
\tikzset{VertexStyle/.style = {%
shape = circle,
shading = ball,
ball color = lgray,
inner sep = 1.5pt
}}
\Vertex[LabelOut,Lpos=270, Ldist=.05cm,x=7,y=4.2,L={2,4}]{24}
\Edge[color=black,style={-triangle 45}](1)(3)
\Edge[style={-triangle 45}](1)(24)
\Edge[style={-triangle 45}](3)(24)

\Vertex[LabelOut,Lpos=0, Ldist=.05cm,x=10,y=5,L={3}]{3}
\tikzset{VertexStyle/.style = {%
shape = circle,
shading = ball,
ball color = black,
inner sep = 1.5pt
}}
\Vertex[LabelOut,Lpos=180, Ldist=.05cm,x=9,y=5]{1}
\Vertex[LabelOut,Lpos=270, Ldist=.05cm,x=9.5,y=4.2,L={2,4}]{24}
\Edge[style={-triangle 45}](3)(1)
\Edge[style={-triangle 45}](3)(24)
\Edge[color=black,style={-triangle 45}](1)(24)

\Vertex[LabelOut,Lpos=90, Ldist=.05cm,x=0.75,y=3]{1}
\Vertex[LabelOut,Lpos=180, Ldist=.05cm,x=0.25,y=2.5]{2}
\Vertex[LabelOut,Lpos=0, Ldist=.05cm,x=1.25,y=2.5]{3}
\tikzset{VertexStyle/.style = {%
shape = circle,
shading = ball,
ball color = lgray,
inner sep = 1.5pt
}}
\Vertex[LabelOut,Lpos=270, Ldist=.05cm,x=0.75,y=2]{4}
\Edge[color=black,style={-triangle 45}](1)(2)
\Edge[color=black,style={-triangle 45}](1)(3)
\Edge[style={-triangle 45}](1)(4)
\Edge[style={-triangle 45}](2)(4)
\Edge[style={-triangle 45}](3)(4)

\Vertex[LabelOut,Lpos=180, Ldist=.05cm,x=2.75,y=2.5]{2}
\Vertex[LabelOut,Lpos=270, Ldist=.05cm,x=3.25,y=2]{4}
\tikzset{VertexStyle/.style = {%
shape = circle,
shading = ball,
ball color = black,
inner sep = 1.5pt
}}
\Vertex[LabelOut,Lpos=90, Ldist=.05cm,x=3.25,y=3]{1}
\Vertex[LabelOut,Lpos=0, Ldist=.05cm,x=3.75,y=2.5]{3}
\Edge[style={-triangle 45}](2)(1)
\Edge[color=black,style={-triangle 45}](1)(3)
\Edge[style={-triangle 45}](1)(4)
\Edge[style={-triangle 45}](2)(4)
\Edge[style={-triangle 45}](3)(4)

\Vertex[LabelOut,Lpos=90, Ldist=.05cm,x=5.75,y=3]{1}
\Vertex[LabelOut,Lpos=180, Ldist=.05cm,x=5.25,y=2.5]{2}
\tikzset{VertexStyle/.style = {%
shape = circle,
shading = ball,
ball color = lgray,
inner sep = 1.5pt
}}
\Vertex[LabelOut,Lpos=0, Ldist=.05cm,x=6.25,y=2.5]{3}
\Vertex[LabelOut,Lpos=270, Ldist=.05cm,x=5.75,y=2]{4}
\Edge[color=black,style={-triangle 45}](1)(2)
\Edge[style={-triangle 45}](3)(1)
\Edge[style={-triangle 45}](1)(4)
\Edge[style={-triangle 45}](2)(4)
\Edge[style={-triangle 45}](3)(4)

\Vertex[LabelOut,Lpos=180, Ldist=.05cm,x=7.75,y=2.5]{2}
\Vertex[LabelOut,Lpos=0, Ldist=.05cm,x=8.75,y=2.5]{3}
\tikzset{VertexStyle/.style = {%
shape = circle,
shading = ball,
ball color = black,
inner sep = 1.5pt
}}
\Vertex[LabelOut,Lpos=90, Ldist=.05cm,x=8.25,y=3]{1}
\Vertex[LabelOut,Lpos=270, Ldist=.05cm,x=8.25,y=2]{4}
\Edge[style={-triangle 45}](2)(1)
\Edge[style={-triangle 45}](3)(1)
\Edge[color=black,style={-triangle 45}](1)(4)
\Edge[style={-triangle 45}](2)(4)
\Edge[style={-triangle 45}](3)(4)

\end{tikzpicture}
}
\end{center}
\caption{\label{fig:stars} The maximal $1$-stars associated with the
``kite graph.'' The edges and vertices participating in each star are in black,
and the rest are grayed out. The $4$-acyclic partitions whose $1$-stars are not
maximal are not shown.}
\end{figure}
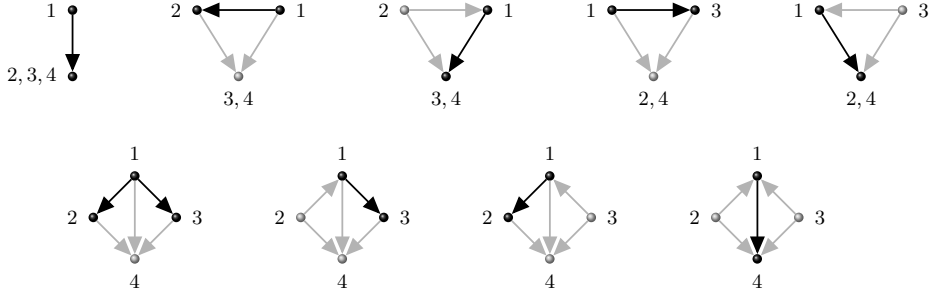

Suppose $\mathcal{S}$ is a maximal $j$-star associated to some $n$-acyclic $(k+1)$-partition
$\mcC'$. We associate to $\mathcal{S}$ a complex $\mcH(\mathcal{S})$ of vectors as follows. Let $m$
be the number of vertices of $\mathcal{S}$, and relabel the vertices of $\mathcal{S}$ with $[m]$ so
that the vertex containing $j$ receives the label $m$. Consider the $G$-parking
function ideal $M_{\mathcal{S}}$ over the ring $\bbK[x_1,\ldots,x_m]$. We define
$\mcH(\mathcal{S})$ to be the complex $\mcF_0(\mathcal{S})_{\langle x_m\rangle}\otimes
\kappa(\langle x_m\rangle)$, with zeroes appended on either end  so that the
$(m-1)\nth$ homological degree of $\mcF_0(\mathcal{S})$ corresponds to the
$k\nth$ homological degree of $\mcH(\mathcal{S})$ (see Example
\ref{maxstar_ex1}).  As a matrix, the map of vectors spaces $\gamma_\ell =
(\delta_{0,\ell})_{\langle x_m\rangle} \otimes \kappa(\langle x_m\rangle)$ is
obtained from the differential map for $\mcF_0(\mathcal{S})$ by replacing all
indeterminates with 1. (The indeterminate $x_m$ never appears in these
differentials.)

To fully specify $\mcF_0(\mathcal{S})$, and hence $\mcH(\mathcal{S})$, over
fields of arbitrary characteristic, we need to define the sign functions
$\sign_{\mcA}$ for every $m$-acyclic $\ell$-partition $\mcA$ of $\mathcal{S}$.
The basis elements for the free module in homological degree $k$ of
$\mcH(\mathcal{S})$ correspond to $j$-stars $\mathcal{S}(\mcC)$ associated
to some $n$-acyclic $k$-partition $\mcC$ refined by $\mcC'$. The contractible
edges of $\mcA$ correspond to $j$-edges of $\mcC$. Thus, we define
$\sign_{\mcA}(e) = \sign_{\mcC}(e)$. Clearly $\sign_{\mcA}(e)$ satisfies
\eqref{eq:sign prop1}, since $\sign_{\mcC}(e)$ does. Now recall that
$\mcF_0(\mathcal{S})$ is exact by Lemma~\ref{thm:tree}. Thus, by
Lemma~\ref{vectred_theo}, $\mcH(S)$ is exact as well.

Let $\mcM = \oplus_{\mathcal{S}} \mcH(\mathcal{S})$, where the direct sum is taken over all
maximal $j$-stars of $G$. Define maps $\phi_k$ from the vector space of
$(\mcF_0)_{P_j}\otimes \kappa(P_j)$ at homological degree $k$ to the corresponding
vector space of $\mcM$ by sending the $n$-acyclic $(k+1)$-partition $\mcC$
to its $j$-star $\mathcal{S}(\mcC)$. These maps are isomorphisms of vector spaces by
construction. In addition, the maps $\phi_k$ commute with the differentials
of $(\mcF_0)_{P_j}\otimes \kappa(P_j)$ and $\mcM$, since edge contractions
in the stars correspond exactly to edge contractions in $G$ giving nonzero
entries in the differentials of $(\mcF_0)\otimes \kappa(P_j)$, and the signs
associated with these contractions are equal.  Hence, we have:

\begin{lemma}\label{vectcomex_theo}
For any integer $j$ from $1,\ldots,n-1$, the complex $(\mcF_0)_{P_j}
\otimes \kappa(P_j)$ is isomorphic as a complex of $\bbK$-vectors spaces to
$\mcM$. In particular, since $\mcH(\mathcal{S})$ is split exact for every maximal
star $\mathcal{S}$, the complex $(\mcF_0)_{P_j}\otimes \kappa(P_j)$ is
split exact as well.
\end{lemma}

Using Lemma~\ref{vectcomex_theo} and Lemma~\ref{vectred_theo}, we complete
the proof of Theorem~\ref{thm:M_G}.

\begin{proof}[Proof of Theorem~\ref{thm:M_G}]
As we observed in the discussion following the statement of
Theorem~\ref{thm:M_G}, $\mcF_0(G)$ is a complex and the cokernel of
$\delta_{0,1}$ is $R/M_G$. By Lemma~\ref{vectcomex_theo}, we know that
$(\mcF_0)_{P_j} \otimes \kappa(P_j)$ is split exact for every integer $j$ with
$1 \le j \le n-1$. It then follows by Lemma~\ref{vectred_theo} that $\mcF_0(G)$
is exact, so that $\mcF_0(G)$ is a free resolution for $M_G$. Furthermore,
since the image of $\delta_{0,k}$ is in $\langle
x_1,\ldots,x_{n-1}\rangle F_{0,k-1}$ for every integer $k$, it follows that
$\mcF_0(G)$ is minimal.
\end{proof}

As a corollary to maximal star decomposition obtained in
Theorem~\ref{vectcomex_theo}, we obtain another formula for Betti numbers of
$M_G$:
\begin{corollary}
Fix an integer $j$ from $1,\dots,n-1$. For integers $0 \le s,t  \le n$, let
$q_{s,t}$ denote the number of maximal $j$-stars of $G$ having $t$
vertices and corresponding to $n$-acyclic $(s+1)$-partitions. Then for every
integer $0 \le k \le n$, we have
$\beta_k(R/I_G)=\beta_k(R/M_G)=\sum_{r=k}^{n-1}\sum_{s=1}^{n-1}q_{r,s}
\binom{s-1}{k}$.
\end{corollary}

We conclude this section with an example of the maximal star decompositions.

\begin{example}\label{maxstar_ex1}
Consider again the ``kite graph'' $G$ depicted in Figure~\ref{fig:kite graph}. 
Let $j=1$. The maximal $j$-stars associated with $G$ are depicted in
Figure~\ref{fig:stars}. The $j$-stars corresponding to each of the $4$-acyclic
$4$-partitions are all maximal. One of these has two edges, and the others have
only one. For the $j$-star $\mathcal{S}$ with two edges, the corresponding
complex $\mcH(\mathcal{S})$ of vector spaces is
\[
0 \rightarrow \bbK^1 \xrightarrow{\Delta_2} \bbK^2 \xrightarrow{\Delta_1} \bbK^1 \rightarrow 0,
\]
where
\[
\Delta_1 = \begin{pmatrix} 1 \\ -1 \end{pmatrix}\qquad \textrm{ and } \qquad \Delta_2 =
\begin{pmatrix} 1 & 1 \end{pmatrix}.
\]
The complexes corresponding to the other maximal stars are $0 \rightarrow
\bbK^1 \rightarrow \bbK^1 \rightarrow 0$. Of the $j$-stars associated with each
of the $4$-acyclic $3$-partitions, only four are maximal, and each of these has
exactly one edge.  Finally, there is exactly one maximal $j$-star associated
with a $4$-acyclic $2$-partitions.

The resulting direct sum decomposition of $(\mcF_0(G))_{P_1} \otimes
\kappa(P_1)$ is the following:
\begin{align*}
 \hspace{-0.1cm} &0 \rightarrow \bbK^4 \rightarrow \bbK^{9} \rightarrow
 \bbK^{6} \rightarrow \bbK^1 \rightarrow 0=\\
  1(&0 \rightarrow \bbK^1 \rightarrow \bbK^2 \rightarrow \bbK^1 \rightarrow 0 \rightarrow 0 ) \\
  \oplus\,\, 3(&0 \rightarrow \bbK^1 \rightarrow \bbK^1 \rightarrow 0 \rightarrow  0 \rightarrow 0) \\
  \oplus\,\, 4(&0 \rightarrow 0 \rightarrow  \bbK^1 \rightarrow \bbK^1 \rightarrow 0 \rightarrow 0) \\
  \oplus\,\, 1(&0 \rightarrow 0 \rightarrow  0 \rightarrow \bbK^1 \rightarrow
  \bbK^1  \rightarrow
  0).
\end{align*} \qed
\end{example}

\section{Exactness of $\mcF_1(G)$: A Gr\"obner Degeneration of Complexes}\label{F1exact_sect}

We know from \cite{CorRosSal02} that $M_G$ is an initial ideal of $I_G$.  From
Gr\"obner basis theory,  there is an integral weight function that realizes
this total order \cite[Chapter 15]{Eisen95}.  As remarked in \cite{ManStu12},
integral weight functions that realize the degeneration from $I_G$ to $M_G$
arise naturally from potential theory on the graph and one such choice is the
function $b_{q}$ studied in \cite{BakSho11}. In this section, we show the
stronger property that the minimal free resolution of $I_G$ also Gr\"obner
degenerates to the minimal free resolution of $M_G$.  As a consequence, $I_G$ and $M_G$
share the coarse Betti numbers. Geometrically,  both $I_G$ and $M_G$ lie in the same
subscheme of the Hilbert scheme that corresponds to varieties that share the
same Betti numbers.

Let $y \in \im_{\bbZ^n}(\Lambda_G)$ have $y_i > 0$ for all $i \ne n$. Let the
integral weight vector $\lambda \in \bbZ^n$ be a solution to the equation
$\Lambda_G \lambda = y$ such that $\lambda_i>0$ for all $i$ from 1 to $n$. Note
that a solution $\lambda$ with these properties is guaranteed since
$(1,\dots,1)$ is in the kernel of $\Lambda_G$.  For example, if we take $y$ to
be an integral multiple of the vector $(1,\dots,1,-(n-1))$ that lies in
$\im_{\bbZ^n}(\Lambda_G)$, the corresponding integral weight vector $\lambda$
is, up to scaling, the vector $b_{q}$ in \cite{BakSho11}.  An important
property of the weight vector $\lambda$, as we will see in
Lemma~\ref{lem:0fiber}, is that the weight $\lambda \cdot D(\mcC)$ of the
divisor corresponding to an $n$-acyclic $k$-partition $\mcC$ is uniquely
maximized among all acyclic $k$-partitions $\mcC'$ that are chip firing
equivalent to $\mcC$, following the analogous result for acyclic orientations
of $G$ in \cite[Theorem 4.14]{BakSho11}.
 
Given an equivalence class $\mfc$ of acyclic $k$-partitions,
let $\eps_\mfc = \lambda\cdot D(\mcC)$, where $\mcC \in \mfc$ is $n$-acyclic
and $D(\mcC)$ is viewed as an element of $\bbZ^n$. Given a monomial $m = \bfx^u
e_{\mfc} \in F_{1,k}$, we define the weight $w(m) = \lambda\cdot u + \eps_\mfc$. 
We define $\mcF_t(G)$ to be the Gr\"obner degeneration, i.\,e.,~the
homogenization, of $\mcF_1(G)$ with respect to the integral weight function
$w$, exactly as described in \cite[Chapter 8, Section 3]{MilStu05}.  More
precisely, $\mcF_t(G)$ is the $(\bbZ^n/\im_{\bbZ^n}\Lambda_G) \times \bbZ$-graded complex
whose $k\nth$ homological degree is the free $R[t]$-module
\[
F_{t,k} = \bigoplus_{\mfc} R[t]\tilde{e}_{\mfc}.
\]
Here, the sum is taken over all equivalence classes $\mfc$ of acyclic
$(k+1)$-partitions, and the basis element $\tilde{e}_{\mfc}$ has degree
$(\eps_\mfc, w(\eps_\mfc))$ for any $\mcC \in \mfc$. The $k\nth$ differential
$\delta_{t,k} : F_{t,k} \to F_{t,k-1}$ is defined by 
\begin{equation}\label{eq:delta^t}
\delta_{t,k}(\tilde{e}_{\mfc}) = t^{\eps_{\mfc}}\sum_f \sign_{\mfc}(f)
t^{-w(m_{\mfc}(f)e_{\mfc/f})}m_{\mfc}(f)\tilde{e}_{\mfc/f}.
\end{equation}
where the sum is taken over all contractible edges $f$ of $\mcC$.  Being the
homogenization of the complex $\mcF_{1}(G)$, the sequence $\mcF_t(G)$ is
automatically a complex of $R[t]$-modules. The condition  $\lambda_i>0$ ensures
a positive grading (as defined in \cite[Chapter 8.3]{MilStu05}). Note that
under the evaluation map $t \mapsto 1$, the complex $\mcF_t(G)$ becomes
$\mcF_1(G)$.  In other words, if we consider $\mcF_t(G)$ as a family of
$R$-complexes, then the fiber over $(t-1)$ is $\mcF_1(G)$.

\begin{example}
For the ``kite graph'' $G$, consider the integral weight vector $\lambda=(5,6,5,2)$ with the indeterminate $t$ having weight two.
The complex $\mcF_t(G)$ is as follows:
\[
0  \rightarrow {R[t]}^4  \xrightarrow{\delta_{t,3}} {R[t]}^9
\xrightarrow{\delta_{t,2}}  {R[t]}^6    \xrightarrow{\delta_{t,1}}   {R[t]}^1
\rightarrow 0.
\]
The matrices of differentials are
\begin{align*}
\delta_{t,1} &=\left(\begin{array}{rrrrrr}
x_{1}^{3}- x_{2} x_{3} x_{4} t  & x_{2}^{2}- x_{1} x_{4} t 
& x_{3}^{2} - x_{1} x_{4} t & x_{1}^{2} x_{2}- x_{3} x_{4}^{2} t^{2} 
& x_{1}^{2} x_{3} - x_{2} x_{4}^{2} t^{2} & x_{1} x_{2} x_{3}- x_{4}^{3} t^{3}
\end{array}\right)\\
\delta_{t,2} &=\left(\begin{array}{rrrrrrrrr}
0 & - x_{2} & - x_{4} t & - x_{3} & - x_{4} t & 0 & 0 & 0 & 0 \\
-  x_{3}^{2} + x_{1} x_{4} t  & - x_{3} x_{4} t & - x_{1}^{2} & 0 & 0 & 0 & 0 & -
x_{4}^{2} t^{2} & - x_{1} x_{3} \\
x_{2}^{2} - x_{1} x_{4} t  & 0 & 0 & - x_{2} x_{4} t & - x_{1}^{2} & -
  x_{4}^{2} t^{2} & - x_{1} x_{2} & 0 & 0 \\
  0 & x_{1} & x_{2} & 0 & 0 & - x_{3} & - x_{4} t & 0 & 0 \\
  0 & 0 & 0 & x_{1} & x_{3} & 0 & 0 & - x_{2} & - x_{4} t \\
  0 & 0 & 0 & 0 & 0 & x_{1} & x_{3} & x_{1} & x_{2}
  \end{array}\right)\\
\delta_{t,3} &=\left(\begin{array}{rrrr}
- x_{4} t & 0 & 0 & x_{1} \\
x_{3} & 0 & - x_{4} t & 0 \\
0 & x_{3} & 0 & x_{4} t \\
- x_{2} & - x_{4} t & 0 & 0 \\
0 & 0 & x_{2} & - x_{4} t \\
x_{1} & x_{2} & 0 & 0 \\
0 & 0 & - x_{1} & x_{2} \\
- x_{1} & 0 & x_{3} & 0 \\
0 & - x_{1} & 0 & - x_{3}
\end{array}\right).
\end{align*}

Under the grading induced by the homogenization, this complex
$\mcF_t(G)$ is the minimal free resolution of the cokernel of
$\delta_{t,0}$.  Substituting $t=1$ in the entries of the differentials gives the
minimal resolution of $I_{G}$ as shown in Example~\ref{IG_ex} and
substituting $t=0$ gives the minimal free resolution of $M_G$ as shown in
Example \ref{MG_ex}.  This property holds for any graph, and we exploit these
properties to prove the exactness of $\mcF_1(G)$.
\qed
\end{example}

For any nonzero $t_0  \in \mathbb{K}$, the fiber $\mcF_{t_0}(G)$ over $(t-t_0)$ is
isomorphic as a graded complex to $\mcF_1(G)$. In particular, the map which sends
$x_i$ to $t_0^{\lambda_i}x_i$ and $e_{\mfc}$ to ${t_0}^{\eps_{\mfc}}e_{\mfc}$
is an isomorphism.  We now show that the fiber over $(t)$ is $\mcF_0(G)$. 
\begin{lemma}\label{lem:0fiber}
$\mcF_t(G)/t\mcF_t(G)$ is isomorphic as an $R$-complex to $\mcF_0(G)$.
\end{lemma}
\begin{proof}
We must show that $t$ divides exactly those monomials of
$\delta_{t,k}(\tilde{e}_{\mfc})$ corresponding to edges $f$ of $\mfc$ which do
not appear in the $n$-acyclic $(k+1)$-partition $\mcA \in \mfc$. 

Let $f$ be a contractible edge of $\mfc$, and let $\mcC \in \mfc/f$ be
$n$-acyclic. Let $\hat{\mcC} \in \mfc$ be the acyclic $k$-partition obtained
from $\mcC$ by introducing the edge $f$ (with its given orientation), as
explained in Remark \ref{binsyz_rem}. Note that $\hat{\mcC}$ is not in general
$n$-acyclic. Now by definition, $m_{\mfc}(f) = \bfx^{D(\hat{\mcC}) - D(\mcC)}$,
and it follows that $w(m_{\mfc}(f)e_{\mfc/f}) = \lambda\cdot D(\hat{\mcC})$.

Thus, it suffices to show that if $\mcA \in \mfc$ is $n$-acyclic and $\mcB \in
\mfc$ is not, then $\lambda\cdot D(\mcA) > \lambda \cdot D(\mcB)$. Note that
the divisor $D(\mcB)$ is obtained from $D(\mcA)$ by a sequence of vertex
firings not including $n$ (c.\,f.~Remark~\ref{rem:edge reversal}). Thus, there
is some nonzero $\sigma \in \bbN^n$ with $\sigma_n = 0$ such that $\sigma
\Lambda_G = D(\mcA) - D(\mcB)$. Hence,
\[
\lambda\cdot(D(\mcA) - D(\mcB)) = \sigma^T \Lambda_G \lambda = \sigma^T\cdot y
> 0
\]
as required. 
\end{proof}

\begin{remark}
Note that Lemma \ref{lem:0fiber} crucially uses the property that the weight of
an $n$-acyclic $k$-partition $\mcC$ is uniquely maximized among acyclic
$k$-partitions $\mcC'$ which are chip firing equivalent to $\mcC$.
\end{remark}

As a result of Lemma~\ref{lem:0fiber}, it follows that $\mcF_t(G)$
is exact.
\begin{theorem}\label{thm:I_G^t}
Let $I_G^t$ denote the homogenization of $I_G$ with respect to the weight
function $w$. The sequence 
\[
\mcF_t(G): 0 \rightarrow F_{t,n-1} \xrightarrow{\delta_{t,n-1}}\cdots 
\xrightarrow{\delta_{t,2}}F_{t,1} \xrightarrow{\delta_{t,1}} F_{t,0}
\]
defined by Equation~\eqref{eq:delta^t} is a minimal free resolution of $R/I_G^t$.
\end{theorem}
\begin{proof}
We have already noted that $\mcF_t(G)$ is a complex. We now show that it is
exact.

Observe that $H_k(\mcF_t(G))\otimes R[t]/(t)$ includes into
$H_k(\mcF_t(G)\otimes R[t]/(t)) = H_k(\mcF_0(G))$. Since the latter is trivial
by Lemma~\ref{lem:0fiber} and Theorem~\ref{thm:M_G}, so is the former.
Consider the short exact sequence
\[
0 \rightarrow \bbK[t] \xrightarrow{t} \bbK[t] \rightarrow \bbK[t]/(t)
\rightarrow 0,
\]
and tensor it over $R[t]$ with $H_k(\mcF_t(G))$. Since $-\otimes H_k(\mcF_t(G))$ is
right-exact, it follows that multiplication by $t$ is a surjection from
$H_k(\mcF_t(G))$ onto itself. By Nakayama's lemma, it follows that
$H_k(\mcF_t(G))$ is trivial, i.\,e., $\mcF_t(G)$ is exact.

As we observed in the proof of Lemma~\ref{lem:0fiber}, firing any set of
vertices not including $n$ produces a divisor with smaller weight with respect
to $\lambda$. It then follows from~\cite[Theorem 14]{CorRosSal02} that the set
\[
\Gamma = \{\delta_{t,1}(\mfc) : \mfc
\textrm{ a chip-firing equivalence class of acyclic 2-partitions}\}
\]
is a Gr\"obner basis for $I_G$ under a monomial order respecting the weight
function $w$. Since the homogenization of a Gr\"obner basis is a Gr\"obner
basis for the homogenization of an ideal, it follows that the homogenization of
$I_G$ is the image of $\delta_{t,1}$, and so $\mcF_t(G)$ is a free resolution
for $R/I_G^t$. Finally, we note that $\mcF_t(G)$ is minimal since the image of
$\delta_{t,k}$ is contained in $\langle x_1,\ldots,x_n\rangle \cdot
F_{t,k-1}$ for all $k$ from 1 to $n-1$.
\end{proof}

We now arrive at the crucial property of Gr\"obner degeneration, from which the
exactness of $\mcF_1(G)$ will follow.
\begin{proposition}[{\cite[Proposition 8.26]{MilStu05}}]\label{prop:flat module}
Let $M$ be a graded submodule of $F_{1,k}(G)$ for some integer $k$, and denote by
$\tilde{M}$ its homogenization with respect to the weight function $w$. Then
$F_{t,k}(G)/\tilde{M}$ is free as a $\bbK[t]$-module.
\end{proposition}

We are finally able to prove Theorem~\ref{thm:I_G}.
\begin{proof}[{Proof of Theorem~\ref{thm:I_G}}]
By Lemma~\ref{lem:I_G complex}, we know that $\mcF_1(G)$ is a complex such that
the cokernel of $\delta_{1,1}$ is $R/I_G$. Since $t-1$ is not a zero-divisor
for $R/I_G^t$ by Proposition~\ref{prop:flat module}, the exactness of
$\mcF_1(G)$ follows from the exactness of $\mcF_t(G)$ (see~\cite[Proposition
8.28]{MilStu05}). As with $\mcF_t(G)$ and $\mcF_0(G)$, the minimality of
$\mcF_1(G)$ is clear: the image of $\delta_{1,k}$ is contained in $\langle
x_1,\ldots,x_n\rangle \cdot F_{1,k-1}$ for all $k$ from 1 to $n-1$.

\end{proof}

\begin{corollary}
The minimal free resolution of $M_G$ is a Gr\"obner degeneration of the minimal free resolution of $I_G$.
\end{corollary}

\begin{remark}
If we homogenize the minimal free resolution of an arbitrary
$\mathbb{Z}$-graded ideal $I$ then the fiber of the resulting complex at $0$
will in general not be exact. Instead, the well known upper semicontinuity
holds \cite[Theorem 8.29]{MilStu05}: the Betti numbers of the ideal $I$ is at
most the corresponding Betti numbers of any initial ideal of $I$.
\end{remark}

\section{CW complex supporting $\mcF_0(G)$}\label{cwcomp_sect}

A free resolution of a monomial ideal is supported by a CW complex if the
differentials of the free resolution are given by an appropriate modification
of the differentials of the CW complex (see \cite{MilStu05}). In general,
minimal free resolutions of monomial ideals are not supported by CW complexes
\cite{Vel08}. However, in this section we will show that $\mcF_0(G)$ is
supported on a CW complex with a fairly simple structure.

\begin{theorem}\label{thm:cellular}
The complex $\mcF_0(G)$ is a cellular resolution, i.\,e., it is supported on
a CW complex.
\end{theorem}

The CW complex supporting $\mcF_0(G)$ has $k$-cells corresponding to the
$n$-acyclic $(k+2)$-partitions of $G$ and has the same poset structure as the
$n$-acyclic partitions under the refinement ordering. In the remainder of this
section, we show that this CW complex is well-defined, and then conclude that
Theorem~\ref{thm:cellular} follows.

Given a graph $G$ on $[n]$, we recursively define an associated cell complex
$\Part(G)$ as follows. We introduce a $0$-cell $e_{\mcC}$ for each $n$-acyclic 
$2$-partition $\mcC$. When the $(k-1)$-skeleton of $\Part(G)$ is
defined, and $\mcC$ is an $n$-acyclic $(k+2)$-partition, define
\[
\partial e_{\mcC} = \cup_{f} \overline{e_{\mcC/f}}.
\]
where the union is taken over all contractible edges $f$ of $\mcC$.
Proposition \ref{prop:cw defined} below shows that $\partial e_{\mcC} \cong
S^{k-1}$.  We introduce a $k$-cell $e_{\mcC}$ for each $n$-acyclic
$(k+2)$-partition $\mcC$ by gluing it to $\partial e_{\mcC}$ along a
homeomorphism with the sphere. Thus, the closure of each $k$-cell in $\Part(G)$
is homeomorphic to the $k$-disk $D^k$.

For any $\overline{e_\mcC}, \overline{e_{\mcC'}} \subset \Part(G)$, we have
$\overline{e_\mcC} \cap \overline{e_{\mcC'}} = \overline{e_{[\mcC,\mcC']}}$, where
$[\mcC,\mcC']$ is the finest $n$-acyclic partition refined by both $\mcC$ and
$\mcC'$. This oriented partition $[\mcC,\mcC']$ is well-defined, and is given
by taking the finest partition refined by $\Pi(\mcC)$ and $\Pi(\mcC')$,
contracting all edges on which the orientations of $\mcC$ and $\mcC'$ do not
agree, and then iteratively contracting all cycles among the remaining oriented
edges.

\begin{definition}
If $\mcC$ is an $n$-acyclic $k$-partition of $G$, and $A$ is a set of edges
appearing in $\mcC$, we say that $A$ is mutually contractible if any subset of
$A$ can be contracted without creating any cycles.
\end{definition}

\begin{proposition}\label{prop:cw defined}
Let $k > 0$. For any $n$-acyclic $(k+2)$-partition $\mcC$, the subcomplex
$\partial e_{\mcC}$ of $\Part(G)$ is homeomorphic to $S^{k-1}$.
\end{proposition}
\begin{proof}
The claim holds for the two connected graphs on three vertices.  We now assume
the claim is true for graphs on at most $n-1$ vertices and let $G$ have $n$
vertices. For $k+2 < n$, the claim holds since for $n$-acyclic
$(k+2)$-partitions $\mcC$, we have $\partial e_{\mcC}$ homeomorphic to the
corresponding subcomplex of $\Part(G_{\Pi(\mcC)})$, and in particular the
$(n-3)$-skeleton $Z$ of $\Part(G)$ is well-defined. Fixing an $n$-acyclic
$n$-partition $\mcC$, it remains to show that $\partial e_{\mcC} \cong S^{n-3}$.

Note that if $e_{\mcC'} \subset \partial e_{\mcC}$ is an $(n-3)$-cell, then $\mcC'$ is
obtained from $\mcC$ by contracting a unique edge. Furthermore, by the
inductive hypothesis, $\overline{e_\mcC'} \cong D^k$ for any $n$-acyclic
$(k+2)$-partition $\mcC'$ with $k+2 < n$. For any nonempty mutually contractible
set of edges $A \subset \mcC$, define $\mcC_A$ as the oriented partition given by
contracting $A$, let $D_A = \overline{e_{\mcC_A}}$, and let
\[
X_A = \cup_{e \in A}D_{\{e\}} \subset \partial e_{\mcC}
\]
Now let $A \subset \mcC$ be mutually contractible, and note that if $B
\subset A$ has $|B| = r>0$ then $D_B \cong D^{n-(r+2)}$. Furthermore, for any
nonempty $B_1,B_2 \subset A$, we have $D_{B_1} \cap D_{B_2} = D_{B_1 \cup
B_2}$.  It follows that if there exists an edge $(u,v) \in A$ for every nonsink
vertex $u$, then $X_A$ has the cell structure of a hollow simplex, so $X_A
\cong S^{n-3}$. Otherwise, $X_A \cong D^{n-3}$.

Note that if $G$ is a tree, then the set $A$ of all edges in $\mcC$ is mutually contractible, and thus
$\partial e_{\mcC} = X_A \cong S^{n-3}$ by the previous paragraph. Thus, we proceed by induction on
the number of edges of $G$.  Let $u \in V$ have at least two out-neighbors in $\mcC$
and be such that for any $v$ with a directed path to $u$, there is a unique
out-neighbor of $v$. Clearly such a vertex $u$ exists whenever $G$ is not a tree because $\mcC$ is
acyclic. An edge of the form $e = (u,v)$ can be legally contracted if and
only if every directed path from $u$ to $v$ contains $(u,v)$. If $e = (u,v)$
cannot be legally contracted, then let $\mcC'$ be the $n$-acyclic $n$-partition
of $G\setminus \{e\}$ which agrees with $\mcC$ on all edges other than $e$.
Then $\partial e_{\mcC}$ is homeomorphic to $\partial e_{\mcC} \subset \Part(G\setminus
\{e\})$. 

Thus, without loss of generality, every out-edge of $u$ can be legally
contracted. In fact the set $A$ of out-edges of $u$ is mutually
contractible, since contracting any subset of $A$ does not create any paths
between the out-neighbors of $u$. 

\vspace{0.5em}
\noindent\emph{Claim.} For any disk $\overline{e_\mcC'} \in \partial
e_{\mcC}\setminus \cup_{f \in A} e_{\mcC_f}$ with $\mcC'$ an $n$-acyclic
$(k+2)$-partition of $G$, we have $\overline{e_\mcC'} \cap X_A \cong D^{k-1}$.
\vspace{0.5em}

Let $\mcC'$ be as in the claim, and let $B \subset A$ be the collection of edges
$(u,v)$ that cannot be legally contracted from $\mcC'$. We have $\partial
e_{\mcC'}$ homeomorphic to the corresponding subcomplex of $\Part(G_\mcC')$,
which is homeomorphic to the corresponding subcomplex of
$\Part(G_\mcC'\setminus B)$. In the latter graph, $A\setminus B$ is mutually
contractible, and $X_{A\setminus B} \cong D^{k-1}$. The claim follows.

We now show that $\partial e_{\mcC} \cong S^{n-3}$ follows from the claim. Let $f \in A$,
let $G' = G\setminus (A\setminus\{f\})$, and let $\mcC'$ be the $n$-acyclic
$n$-partition of $G'$ corresponding to $\mcC$.  We have $\partial e_{\mcC'} \cong
S^{n-3}$ by the inductive hypothesis since $|A| \ge 2$. On the other hand, note
that for any $\mcA$ refined by $\mcC'$, if $f$ appears in $\mcA$ then $f$
is legally contractible in $\mcA$. This follows from the fact that $u$ has
a unique out-neighbor in $\mcC'$ and the induced subgraph on all vertices with
a path to $u$ is a tree. Thus, for any $\overline{e_\mcA} \subset
\partial e_{\mcC'}\setminus e_{\mcC'_f}$ with $\mcA$ an $n$-acyclic $(k+2)$-partition
of $G'$, we have $D_f \cap \overline{e_\mcA} \cong D^{k-1}$. On the other hand,
if $\mcA$ and $\mcB$ are refined by $\mcC$ and do not join $u$ to any of its
out neighbors, then the corresponding acyclic partitions of $G'$ are both
refined by $\mcC'$ as well, and $[\mcA,\mcB]$ is the same in $G$ as in $G'$. It
then follows from the claim that we have a homeomorphism $\partial e_{\mcC} \cong
\partial e_{\mcC'}$ given by mapping $X_A$ to $D_f$ and preserving the other cells.
\end{proof}

We remark that when $G$ is saturated, each $n$-acyclic $k$-partition $\mcC$ has
exactly $k-1$ contractible edges, which are in fact mutually contractible. In
that case, $\Part(G)$ has the cell structure of a simplicial complex, and in
fact it is the Scarf complex described by Postnikov and Shapiro in
\cite{PosSha}.

\begin{proof}[Proof of Theorem~\ref{thm:cellular}]
We show that after appropriately labeling the cells of $\Part(G)$ with
monomials in $R$, the resulting cellular complex is isomorphic to
$\mcF_0(G)$ (cf. \cite[Chapter 4]{MilStu05}).

Label each cell $e_\mcC$ of $\Part(G)$ with the monomial $\bfx^{D(\mcC)}$.  We
claim that the label of $e_\mcC$ is the least common multiple of the labels of
the cells in $\partial e_{\mcC}$, so that $\Part(G)$ is a \emph{labeled CW
complex}. Certainly $\bfx^{D(\mcC/e)}$ divides $\bfx^{D(\mcC)}$ for every
contractible edge $e$ of $\mcC$. On the other hand, if $\mcC$ is an $n$-acyclic
$k$-partition for some $k\ge 3$, then $\mcC$ then for every vertex $j$ of $G$,
there is a contractible edge $e$ of $\mcC$ such that $j \notin e^-$. Then
$(D(\mcC/e))_j = (D(\mcC))_j$, and the claim follows.  Identifying the standard
basis elements of the cellular complex for $\Part(G)$ with the standard basis
elements of $\mcF_0(G)$, we see that the cellular monomial matrices
associated with $\Part(G)$ are exactly the differential maps of
$\mcF_0(G)$, up to signs. Since the cellular monomial matrices associated with
$\Part(G)$ form a complex, if we define the sign function $\sign_{\mfc}$ of
$\mcF_0(G)$ to be the sign of the corresponding entry from the cellular complex
for $\Part(G)$, we see that $\sign_{\mfc}$ satisfies the required
Property~\eqref{eq:sign prop1}. Thus, under an appropriate sign function,
$\mcF_0(G)$ is supported on $\Part(G)$.
\end{proof}

\section{Conclusion}

In this paper, we explicitly constructed minimal free resolutions of the ideals
$M_G$ and $I_G$. The general version of this approach is to associate a
combinatorial object (such as a graph or a simplicial complex) with a graded
module and attempt to describe the minimal free resolution of the graded module in
terms of the underlying combinatorial structure. It is natural to ask whether
directed graphs are suitable for this purpose, i.\,e., whether the techniques of
this paper generalize to toppling ideals and $G$-parking function ideals of
directed graphs $G$. As we noted in the introduction, any lattice ideal coming
from a full rank submodule of the root lattice can be realized as the Laplacian
lattice ideal of a directed graph. Thus, the question of finding minimal free
resolutions of such lattice ideals is both challenging and exciting.  

\bibliography{minfree}{}
\bibliographystyle{amsplain}

\end{document}